\documentclass{amsart}
\usepackage{amsmath}%
\usepackage{amsfonts}%
\usepackage{amssymb}%
\usepackage{graphicx}
\newtheorem{theorem}{Theorem}[section]

\newtheorem{fact}[theorem]{Fact}

\newtheorem{corollary}[theorem]{Corollary}

\newtheorem{definition}[theorem]{Definition}

\newtheorem{lemma}[theorem]{Lemma}

\newtheorem{proposition}[theorem]{Proposition}
\newtheorem{remark}[theorem]{Remark}

\numberwithin{equation}{section}
\def\Xint#1{\mathchoice
{\XXint\displaystyle\textstyle{#1}}%
{\XXint\textstyle\scriptstyle{#1}}%
{\XXint\scriptstyle\scriptscriptstyle{#1}}%
{\XXint\scriptscriptstyle\scriptscriptstyle{#1}}%
\!\int}
\def\XXint#1#2#3{{\setbox0=\hbox{$#1{#2#3}{\int}$}
\vcenter{\hbox{$#2#3$}}\kern-.5\wd0}}

\def\dashint{\Xint-}

\newcommand{\R}{\mathbb{R}}
\newcommand{\Q}{\mathbb{Q}}
\newcommand{\N}{\mathbb{N}}

\newcommand{\M}{\mathcal{M}}
\newcommand{\Ha}{\mathcal{H}}

\newcommand{\Ne}[2]{N^{1,#1}(#2)}

\newcommand{\Nem}[3]{N^{1,#1}(#2;#3)}
\newcommand{\Wem}[3]{W^{1,#1}(#2;#3)}
\newcommand{\Neml}[3]{N^{1,#1}_{loc}(#2;#3)}
\newcommand{\Di}[3]{D^{1,#1}(#2;#3)}
\newcommand{\Dir}[2]{D^{1,#1}(#2)}

\newcommand{\diaco}[1]{\widehat{#1}_{\operatorname{diag}}}

\newcommand{\Mod}{\operatorname{Mod}}
\newcommand{\dist}{\operatorname{dist}}
\newcommand{\diam}{\operatorname{diam}}

\newcommand{\Lip}{\operatorname{Lip}}
\newcommand{\LIP}{\operatorname{LIP}}
\newcommand{\ud}{\mathrm {d}}

\newcommand{\Cp}{\operatorname{Cap}}

\begin{document}
\title{Homotopy classes of Newtonian spaces}
\thanks{The author was supported by the Academy of Finland, project no. 1252293, and the V\"ais\"al\"a Foundation.\newline 
\noindent \emph{Email}: elefterios.soultanis@helsinki.fi}
\author{Elefterios Soultanis}
\date{\today}

\address{P.O. Box 68 (Gustaf H\"allstr\"omin katu 2b),\newline
\indent FI-00014 University of Helsinki}
\email{elefterios.soultanis@helsinki.fi}%
\urladdr{http://wiki.helsinki.fi/display/mathstatHenkilokunta/Soultanis\%2C+Elefterios}

\begin{abstract}
We study notions of homotopy in the Newtonian space $\Nem pXY$ of Sobolev type maps between metric spaces. After studying the properties and relations of two different notions we prove a compactness result for sequences in homotopy classes with controlled homotopies.
\end{abstract}

\maketitle
\tableofcontents

\section{Introduction}

Interest in homotopy classes of mappings and energy minimizers arises naturally both in the theory of PDE's -- where certain energy minimizers in homotopy classes provide natural examples of non-uniqueness of some systems of partial differential equations (see \cite{bre03}) -- and in the study of the geometry of manifolds. Minimizing some energy in a given homotopy class provides one with a well-behaved representative of that class. Topological conclusions from the study of harmonic maps in given homotopy classes were drawn, for instance, by R. Schoen and S.T. Yau in \cite{sch76, sch79, sch97}. For $p$-harmonic maps (with $1<p<\infty$), connections to higher homotopy groups, as well as to homotopy classes of maps arise, see e.g. \cite{wei98, pig09, ver12}.


From early on in the work of various authors, such as Eells and Sampson \cite{eel64}, it has been noted that certain methods of obtaining existence results for harmonic maps in homotopy classes are restricted to the setting of non-positively curved target manifolds (see the survey article \cite{eel78} for further discussion). Some results, such as in the papers \cite{bur84, whi88} of Burstall and White, have been obtained for the existence, regularity (and, more rarely, uniqueness) of harmonic and $p$-harmonic maps between general Riemannian manifolds, with varying assumptions. More recently $p$-harmonic maps between general Riemannian manifolds have been studied in \cite{pig09, hol11, ver12} to mention but a few.

Towards a nonsmooth theory the assumption of (some sort of) nonpositive curvature on the target space seems to become compulsory. Starting with Gromov's and Schoen's work \cite{gro92}, continued in \cite{kor93} a theory of harmonic maps from a Riemannian manifold (or Riemannian polyhedron in \cite{eel01}) to a nonpositively curved metric space (in the sense of Alexandrov, see Section 1.2 below) was built. Jost, in a series of papers \cite{jost94, jost96, jost97} studied harmonic maps from metric spaces with a doubling measure and a Poincar\'e inequality to metric spaces of nonpositive curvature. This setting is closest to ours; with basically the same assumptions we proceed to define and study homotopy classes using tools coming from analysis in metric spaces (more of which in Section 1.1 below).

\subsection{Main results}
The present paper may be divided into two parts. In the first we focus mainly on general properties of homotopies in the setting Sobolev type maps between metric spaces. The second part is concerned with compactness properties and existence of energy minimizers in homotopy classes. A new approach for proving these in the metric setting is proposed but we are unable to complete it.

We work in the setting of metric spaces; the domain $(X,d,\mu)$ is always assumed to be a complete metric space with a doubling measure supporting a weak $(1,p)$-Poincar\'e inequality (see Definition \ref{Poincare}), and the target $(Y,d)$ a complete separable geodesic space.

In the second part we will assume, in addition, that the target is a locally convex space (see Definition \ref{locconv}).

\subsubsection*{First part: Sections 1-4.} More precisely, in the first part, we work in the framework of the Dirichlet classes $\Di pXY$ of maps between metric spaces, with $p\in (1,\infty)$. We introduce two distinct topologies on $\Di pXY$, and two notions of homotopy for maps in $\Di pXY$. The precise definitions are given in Subsection 1.2 (Definitions \ref{standard} and \ref{ohtatop}) and in Subsection 2.1 (Definitions \ref{phomotopy} and \ref{path}), respectively, but we briefly explain the main idea of the definitions here.

\bigskip\noindent\emph{Topologies.}The standard topology refers to the topology on $\Di pXY$ induced by the seminorms \[ \|u\|_{\Di p{\Omega}{\ell^\infty(Y)}}:=\left(\int_\Omega\|u\|^p\ud\mu+\int_Xg_u^p\ud\mu\right)^{1/p} \] where $\Omega$ ranges over the domains of $X$ with compact closure. The \emph{Ohta topology} on $\Di pXY$ is in turn induced by the family of pseudometrics \[d_\Omega(u,v)=\int_\Omega d_Y(u,v)\ud\mu+\left(\int_X|g_u-g_v|^p\ud\mu\right)^{1/p}, \] $\Omega$ ranging again over domains of $X$ with compact closure.

In both cases, if $X$ is compact, the above expressions with $\Omega=X$ yield a metric which generates the topologies.

\bigskip\noindent\emph{Homotopies.} Two maps in the Dirichlet class $\Di pXY$ are said to be \emph{path homotopic} if they can be connected with a continuous path in $\Di pXY$, while we say they are \emph{$p$-quasihomotopic} if they are homotopic outside sets of arbitrarily small $p$-capacity. The two different topologies on $\Di pXY$ give rise to two distinct notions of path homotopy. See Definitions \ref{phomotopy} and \ref{path}.

\bigskip\noindent The first part of the paper is largely devoted to the relationship between the two notions of homotopy. We summarise the main findings below.

\begin{theorem}\label{ppath}
$(X,d,\mu)$ is a complete space with a doubling measure $\mu$ supporting a weak $(1,p)$-Poincar\'e inequality, and $Y$ a separable complete locally convex metric space. If two maps $u,v\in\Di pXY$ are $p$-quasihomotopic they may be connected by continuous path in the Ohta topology of $\Di pXY$.
\end{theorem}
\begin{theorem}\label{reimpliesp}
Suppose $(X,d,\mu)$ is a complete space with a doubling measure $\mu$ supporting a weak $(1,p)$-Poincar\'e inequality, and $Y$ a separable complete geodesic space. Let $u,v\in \Di pXY$ and $h:[0,1]\to \Di pXY$ be a map connecting $u$ and $v$ (i.e. $h(0)=u, h(1)=v$). Suppose that there exists a constant $C$ and, for every compact $K\subset X$ some $C_K\in (0,\infty)$ so that \[ \|g_{d_Y(h_t,h_s)}\|_{L^p(X)}\le C|t-s|\textrm{ and }\int_Kd_Y(h_t,h_s)\ud\mu \le C_K|t-s| \] for all $t,s\in [0,1]$. Then $u$ and $v$ are $p$-quasihomotopic.
\end{theorem}
In particular we have the following corollary.
\begin{corollary}\label{rect}
Let $X$ and $Y$ be as in Theorem \ref{reimpliesp} and assume in addition that $X$ is compact. Suppose that $u,v\in \Nem pXY$ can be joined by a rectifiable curve in $\Nem pXY$ in the standard metric. Then $u$ and $v$ are $p$-quasihomotopic.
\end{corollary}
For Riemannian manifolds we have the following result.
\begin{theorem}\label{pquasi}
Let $M,N$ be smooth compact Riemannian manifolds. If two maps $u,v\in W^{1,p}(M;N)$ are $p$-quasihomotopic then they are path-homotopic.
\end{theorem}
The converse does not necessarily hold, as is demonstrated by an example in Section 4.

Concerning the different topologies we have the following result.
\begin{proposition}\label{ohta}
Let $M,N$ be smooth compact Riemannian manifolds. Then $\Di pMN=\Wem pMN$ and the standard and Ohta topologies on $\Wem pMN$ agree.
\end{proposition}

\subsubsection*{Second part: Sections 5-6.} The initial main goal in this paper was to prove the stability of $p$-quasihomotopy classes under $L^p$-convergence, in the spirit of \cite{whi88}. Such a result would imply existence of energy-minimizing maps in a given homotopy class.

Indeed, the latter part of the paper develops a new approach for establishing the stability result. Existence of energy minimizers in homotopy classes has been studied in \cite{whi88, kor93, fug08, fug08'} as well as in \cite{bre01, han03} but the methods in these papers are specific to the manifold setting. In particular the works \cite{fug08,fug08'} fill in a gap in the proofs of existence of energy minimizers in \cite{eel01}.

Our approach is based on the following characterization of $p$-quasihomotopy in terms of lifts to the diagonal cover (see Section 4 for the definition of the diagonal cover).
\begin{theorem}\label{lifting}
Let $(X,d,\mu)$ be a complete space with a doubling measure $\mu$ supporting a weak $(1,p)$-Poincar\'e inequality, and suppose $Y$ is a separable complete locally convex space. Two maps $u,v\in \Di pXY$ are $p$-quasihomotopic if and only if the product map $(u,v)\in \Di pX{Y\times Y}$ admits a lift $h\in \Di pX{\diaco Y}$ with respect to the diagonal covering map  $\phi:\diaco Y\to Y\times Y$.
\end{theorem}
\noindent See also \cite[Theorem 1.2]{teri2}.

\bigskip\noindent To describe our approach fix a map $v\in \Di pXY$, with both $X$ and $Y$ compact. Denote the set of maps $u\in \Di pXY$ $p$-quasihomotopic to $v$ by $[v]$. The covering $\phi=:(\phi_0,\phi_1):\diaco Y\to Y\times Y$ induces a map $\overline\phi:\Di pX{\diaco Y}\to \Di pXY$ given by $$\overline\phi(h)=\phi_1\circ h.$$ By theorem \ref{lifting} $u\in [v]$ if and only there is a map $h\in \Di pX{\diaco Y}$ such that $\phi\circ h=(u,v)$. Therefore $\phi$ restricts to a map $\overline\phi:H^v\to [v]$ where $$H^v=\{ h\in \Di pX{\diaco Y}: \phi_0\circ h=v \}.$$

The strategy is to view the restricted map as a covering map and $H^u$ as a covering space. Under the appropriate technical assumptions the stability result would follow from the fact that $H^v$ is a proper metric space (proven in section 5), $[v]$ is known to be precompact (the Rellich Kondrakov theorem) and $\overline\phi$ is a covering map. \emph{However, I have been unable to prove this last part,} and this inability comes from a lack of knowledge concerning the metric geometry of the space $\Nem pXY$.

The problem, which is essentially the existence of a convergent subsequence of lifts, stems from the same difficulty that is present in \cite{eel01} (corrected in \cite{fug08, fug08'}). The numerous details of this (ultimately failed) attempt are presented in Sections 5 and 6.


\bigskip\noindent It nevertheless seemed reasonable to communicate the partial results obtained along the way, in hope of encouraging future research for a better understanding of the metric properties of Newtonian classes of maps and for the existence of minimizers of a suitable energy in homotopy classes in this general setting.

\vspace{0.5cm}
\subsection*{Outline} The paper is organized as follows.

\subsubsection*{Section 1} In the first and second subsections on the introduction, relevant facts on analysis on metric spaces are presented. Subsection 1.2 contains the definitions of the Newtonian and Dirichlet classes and the standard and Ohta topologies. Poincar\'e inequalities and $p$-quasicontinuity are presented in Subsection 1.3. The fourth subsection serves as a brief review of the basics of nonpositively curved spaces. Both the definition of Alexandrov and that of Busemann are presented and briefly discussed.


\subsubsection*{Section 2} In the second section we focus on two different notions of homotopy, the definitions of each being given in Subsection 2.1. Some properties of each are exhibited and the relationship between the different notions is studied. Theorem \ref{ppath} follows immediately from Theorems \ref{geodhom} and \ref{path}, as explained in Remark \ref{rem} in Subsection 2.2. The proof of Theorem \ref{reimpliesp} is given in Subsection 2.3.

\subsubsection*{Section 3} The third section is devoted to ''lifts'' of homotopies. The construction and some properties of the diagonal covering map $\phi:\diaco Y\to Y\times Y$ is given in Subsection 3.1. Subsection 3.2 contains the proof of Theorem \ref{lifting} in two parts, Propositions \ref{target} and \ref{converse}, see Remark \ref{repath}.

\subsubsection*{Section 4} The fourth section covers the manifold case, recalling the necessary concepts of $([p]-1)$-homotopy and some useful lemmas. The proof of Theorem \ref{pquasi} is presented in Subsection 4.2, as well as a counterexample to the converse statement of Theorem \ref{pquasi}, and the brief proof of Proposition \ref{ohta}.

\subsubsection*{Sections 5 and 6} The fifth and sixth sections form the second part of the paper, concerned with the stability of $p$-quasihomotopy classes under $L^p$-convergence.In the fifth section the details of the approach to the stability of $p$-quasihomotopy classes are explained, while the sixth section presents a weak compactness result and discusses some open problems and possible future research directions.


\subsubsection*{Notation and convention} Throughout this paper, the notation $$f_A=\dashint_A f\ud\mu := \frac{1}{\mu(A)}\int_A f\ud\mu$$ will be used for the \emph{average} of a locally $\mu$-integrable function $f$ over the $\mu$-measurable set $A$, with positive measure. The centered maximal function is denoted by $$ \M_R f(x):=\sup_{0<r<R}\dashint_{B(x,r)}f\ud\mu.$$ For a number $\sigma>0$, the \emph{dilated ball} $\sigma B$ of a (open or closed) ball $B=B(x,r)$ is $$\sigma B = B(x,\sigma r).$$

\bigskip\noindent The \emph{length} of a path $\gamma$ joining two points $x,y\in Z$ in a metric space is the following: \[ \ell(\gamma)= \sup\{\sum_{k=1}^nd_Z(\gamma(a_k),\gamma(a_{k-1})): a=a_0<a_1<\ldots <a_n=b  \}. \] In general, this quantity may be infinite. Paths $\gamma$ for which $\ell(\gamma)<\infty$ are called rectifiable. A rectifiable path $\gamma$ can always be \emph{affinely reparametrized} so that $\gamma: [0,1]\to Y$ and $d(\gamma(t),\gamma(s))\le \ell(\gamma|_{[t,s]})=\ell(\gamma)|t-s|$ for all $t,s\in [0,1],\ t<s$; see \cite[Proposition 2.2.9]{pap05}. We will call this the constant speed parametrization of a rectifiable path $\gamma$.

If not otherwise stated, we will always regard rectifiable curves $\gamma$ in a metric space $Z$ as being maps $\gamma:[0,1]\to Z$.

\subsection{Upper gradients, Newtonian and Dirichlet classes of maps}


\bigskip\noindent A metric measure space is a locally compact metric space $(X,d)$ equipped with a Borel regular measure $\mu$ with the property that $0<\mu(B)<\infty$ for all open balls $B\subset X$. 

We say that the metric measure space is \emph{doubling} if the measure is doubling, i.e. there is a constant $0<C<\infty$ such that \[ \mu(B(x,2r))\le C\mu(B(x,r)) \] for all balls $B(x,r)\subset X$ with $r<\diam X$. Note the difference to saying that a \emph{metric space} is doubling, which means that for some fixed number $N$, any ball can be covered with at most $N$ balls of half the radius. Note that  these are distinct notions; for instance $\Q$ is a doubling metric space but does not support a doubling measure. For details on the relationship of the two notions, see \cite{haj03, hei01}.


Let $u:X\to Y$ be a map between metric spaces. A non-negative Borel function $g:X\to [0,\infty]$ is said to be an upper gradient of $u$ if, for every rectifiable curve $\gamma$ with endpoints $x$ and $y$ we have the inequality 
\begin{equation}\label{ug}
d_Y(u(x),u(y))\le \int_\gamma g \ud s.
\end{equation}

The \emph{$p$-modulus of a path family} $\Gamma$ is defined as \[ \Mod_p(\Gamma)=\inf\left\{ \int_X\rho^p\ud\mu :\rho\ge 0\textrm{ Borel and }\int_\gamma\rho\ge 1 \ \forall \gamma\in\Gamma \right\}. \]
A family $\Gamma$ of rectifiable curves $\gamma:[a,b]\to X$ has zero $p$-modulus if there exists a non-negative Borel function $h\in L^p(\mu)$ so that $$ \int_\gamma h=\infty \quad \forall \gamma\in\Gamma.$$ See \cite{haj03, hei00, hei01} for the definition of the path-integral.

If $u:X\to Y$ is a map and $g:X\to [0,\infty]$ a Borel function so that (\ref{ug}) is satisfied for all curves \emph{except a curve family that has zero $p$-modulus} we say that $g$ is a $p$-weak upper gradient of $u$.

\bigskip\noindent Upper gradients and their $p$-weak counterparts enable us to define a concept of $p$-capacity of subsets of $X$, analogously with the classical $p$-capacities. Let $(X,d,\mu)$ be a metric measure space, $E\subset X$ a subset and $p\ge 1$. The Sobolev $p$-\emph{capacity} of the set $E$ is defined by \[\Cp_p(E)= \inf\{\int_X(|u|^p+g^p)\ud\mu: g\textrm{ an upper gradient for }u\textrm{ s.t. } u\ge 1\textrm{ on }E \}. \] Let (P) be a defined pointwise property. We say that (P) holds $p$-quasieverywhere if the set where (P) fails to hold has $p$-capacity zero.

A \emph{condenser} is a pair of subsets $(E,\Omega)$ where $E\subset \Omega$ and $\Omega$ is open. The $p$-capacity of a condenser is defined by \[\Cp_p(E;\Omega)= \inf\{\int_\Omega g^p\ud\mu: u\ge 1\textrm{ on }E, \ u=0\textrm{ on } X\setminus\Omega \}, \] where $g$ is an upper gradient of $u$. As we shall see this concept will play an important role for us. More information on $p$-capacities, equivalent notions and variants, can be found for instance in \cite{kin00, bjo11}.

\subsubsection*{Maps with (locally) integrable upper gradients} To study maps between metric spaces we adopt the framework used in \cite{hei00}. Let $(X,d,\mu)$ be a metric measure space and $V$ a Banach space with the Lipschitz extension property; that is, given any metric space $Z$, each $L$-Lipschitz map $f:A\to V$ from an arbitrary subset $A\subset Z$ may be extended to a $CL$-Lipschitz map $\overline f:X\to V$, with constant $C$ independent of $Z,A$ and $f$. Examples of such spaces are $V=\R$ and $V= \ell^\infty$.

A map $u:X\to V$ is \emph{measurable} if $u^{-1}(U)$ is measurable for every open set $U\subset V$. It is essentially separably valued if there is a set $N\subset X$ with $\mu(N)=0$ so that $u(X\setminus N)$ is a separable subset of $V$. See \cite[Chapter 2]{hei00} for a detailed discussion of Banach space valued $L^p$-spaces.


The \emph{Dirichlet class} $\Di pXV$ consists of measurable maps $u:X\to V$ which have a $p$-integrable $p$-weak upper gradient $g$.\footnote{It is implicitly understood that maps $u,v$ which agree outside a set of $p$-capacity zero are identified, similarly to the usual $L^p$-theory.} Since for any $p$-weak upper gradient $g$ of $u$ (not necessarily $p$-integrable) there is a sequence $g_k$ of \emph{upper gradients} such that $\|g_k-g\|_{L^p(X)}\to 0$, \cite[Lemma 1.46]{bjo11}, it follows that the requirement of $u$ having a $p$-integrable $p$-weak upper gradient is equivalent to requiring that $u$ has a $p$-integrable upper gradient.

\subsubsection*{Minimal $p$-weak upper gradients} As in \cite[Section 2.2]{bjo11} (or \cite[Chapters 7 and 8]{HKST07}) it can be seen that the set \[ G_u=\{ g\in L^p(X): g\textrm{ is a $p$-weak upper gradient for u} \} \] is a closed and convex lattice, if $p>1$. It follows that there is a unique minimal element $g_u$ in the sense that for all $g\in G_u$, one has $g_u\le g$ almost everywhere. We arrive at the following \cite[Theorem 2.5]{bjo11}.

\begin{theorem}
For $p>1$, every map $u\in \Di pXV$ has a unique minimal $p$-integrable $p$-weak upper gradient, denoted $g_u$.
\end{theorem}

\subsubsection*{(local) Newtonian classes} We say that a map $u:X\to V$ belongs to the \emph{local Newtonian class}, $\Neml pXV$, if $u$ is locally $p$-integrable and possesses a ($p$-weak) upper gradient $g\in L^p_{loc}(X)$, while the \emph{Newtonian class}, $\Nem pXV$, consists of maps $u\in L^p(X;V)$ with a ($p$-weak) upper gradient $g\in L^p(X)$. 

\bigskip\subsubsection*{Maps with metric space target} Let $Y$ be a complete metric space. Recall the Kuratowski embedding $Y\to \ell^\infty(Y)$ where we send a point $y\in Y$ to the function $d_y-d_e$. Here $e\in Y$ is a fixed point and $d_y(x):=d(x,y)$. We define the classes 
\begin{align*}
&\Neml pXY = \{ u\in\Neml pX{\ell^\infty(Y)}: u(x)\in Y\textrm{ for $p$-quasievery }x\in X \} \\
&\Di pXY = \{ u\in\Di pX{\ell^\infty(Y)}: u(x)\in Y\textrm{ for $p$-quasievery }x\in X \} \\
&\Nem pXY = \{ u\in\Nem pX{\ell^\infty(Y)}: u(x)\in Y\textrm{ for $p$-quasievery }x\in X \}.
\end{align*}
We will mainly concern ourselves with $\Di pXY$.

\begin{definition}\label{standard}
The family of seminorms \[ \|u\|_{\Di p{\Omega}{\ell^\infty(Y)}}^p := \int_\Omega\|u\|^p\ud\mu+\int_Xg_u^p\ud\mu, \] for domains $\Omega\subset X$ with compact closure gives rise to a topology on $\Di pX{\ell^\infty(Y)}$. The restriction of this topology to $\Di pXY$ is called the \emph{standard topology} on $\Di pXY$.
\end{definition}
This way $\Di pXY$ becomes a closed subspace of $\Di pX{\ell^\infty(Y)}$. Clearly $u_j\to u$ as $j\to \infty$ in the standard topology if and only if \[ \|u-u_j\|_{\Di p{\Omega}{\ell^\infty(Y)}}\to 0\textrm{ as }j\to\infty \] for all domains $\Omega\subset X$ with compact closure.

There is also a different topology we may put on $\Di pXY$. We define it next.

\begin{definition}\label{ohtatop}
The topology on $\Di pXY$ induced by the family of pseudometrics \[ d_\Omega(u,v)=\int_\Omega d_Y(u,v)\ud\mu + \left(\int_X|g_u-g_v|^p\ud\mu\right)^{1/p}, \] with domain $\Omega\subset X$ with compact closure, is called the \emph{Ohta topology}.
\end{definition}
The topology defined above is based on the notion used by Ohta in \cite[Section 4]{oht04} (with $\|d(u,v)\|_{L^p(\Omega)}$ replaced by $\|d(u,v)\|_{L^1(\Omega)}$). Since $|g_u-g_v|\le g_{u-v}$ almost everywhere we see that \emph{convergence in the standard topology implies convergence in the Ohta topology.}

\bigskip\noindent Note that if $X$ is compact the expressions $\|u\|_{\Di p{\Omega}V}$ and $d_\Omega$ for $\Omega=X$ define metrics on $\Di pXY = \Nem pXY$. In such a case we call these the standard and Ohta metric, respectively. In general, the Ohta metric is \emph{ not} complete, see Subsection 4.2.

\subsection{Poincar\'e inequalities and its consequences} An analytic way of imposing a condition that ties the (geo)metric properties of $X$ and the behaviour of the measure $\mu$ is to require that upper gradients also control the behaviour of maps in some \emph{integral average} sense. This is done by the Poincar\'e inequality. 

We say that a metric measure space $(X,d,\mu)$ supports a weak $(1,p)$-Poincar\'e inequality if, whenever $u:X\to \R$ is locally integrable and $g:X\to [0,\infty]$ is a locally integrable upper gradient of $u$ the inequality 
\begin{equation}\label{Poincare}
\dashint_{B}|u-u_B|\ud\mu\le C\diam(B)\left(\dashint_{\sigma B}g^p\ud\mu\right)^{1/p}
\end{equation}
is satisfied with constants $C,\sigma$ independent of $u,g$ and $B$. The constants $\sigma, C$ in the Poincar\'e inequality and the doubling constant of the measure will be referred to as the \emph{data} of the space $X$.

By now doubling metric measure spaces supporting a weak $(1,p)$-Poincar\'e inequality are known to enjoy many geometric as well as analytic properties. We will only mention some of these that are relevant to this paper. There are numerous sources on the subject, and the interested reader is referred to \cite{hei98, kos98, che99, hei00, sha00, haj00, hei01, kei02, haj03, HKST07, bjo11} to name a few.

We record the following useful theorem from \cite[Theorem 4.3]{hei00}.

\begin{theorem}
Suppose $(X,d,\mu)$ is a complete doubling metric measure space. Then $X$ supports a weak $(1,p)$-Poincar\'e inequality for $p>1$ if and only if it supports a weak $(1,p)$-Poincar\'e inequality for $V$-valued maps, for any Banach space $V$, i. e. if there are constants $C',\sigma'\in [1,\infty)$ such that for every locally integrable map $u:X\to V$ and every upper gradient $g$ of $u$ the inequality \[ \dashint_B\|u-u_B\|_V\ud\mu \le C' r \left(\dashint_{\sigma' B}g^p\ud\mu\right)^{1/p} \] holds for all balls $B=B(x,r)$. The constants $ C'$ and $\sigma'$ then depend only on $p$ and the data of $X$.
\end{theorem}

\medskip\subsubsection*{Measurability and local $p$-integrability} In the definition of $\Di pXV$ no local integrability assumption is made for a function $u\in \Di pXV$ itself. However the inequality \[ |\|u(\gamma(1))\|_V-\|u(\gamma(0))\|_V|\le \|u(\gamma(1))-u(\gamma(0))\|_V\le \int_\gamma g\] remains true without any measurability assumptions and implies that a $p$-weak upper gradient $g$ of $u$ is also a $p$-weak upper gradient for the function $x\mapsto \|u(x)\|_V$. 


From \cite[Theorem 1.11]{jar07} we have the following.

\begin{theorem}\label{rmk}
If $(X,d,\mu)$ supports a weak $(1,p)$- Poincar\'e inequality and $f:X\to [-\infty,\infty]$ is a function that has a $p$-integrable upper gradient, then $f$ is measurable and locally $p$-integrable.
\end{theorem}
Applying this to $f=\|u\|_V$ we see that in fact $x\mapsto \|u(x)\|_V$ is measurable (see \cite{hei00}) and locally $p$-integrable. Thus,

\medskip\noindent\emph{if $u$ is essentially separably valued, the existence of a $p$-integrable ($p$-weak) upper gradient implies both measurability and local $p$-integrability of $u$.}


\begin{remark}\label{lipup}
Regarding the minimal $p$-weak upper gradient of a locally lipschitz map $u:X\to V$, with $(X,d,\mu)$ a doubling metric measure space supporting a weak $(1,p)$-Poincar\'e inequality, we note that there is a constant $C$ depending only on the data of $X$ such that \[ \Lip u\le Cg_u \] almost everywhere. To see this we simply note that Keith's proof of his result \cite[Proposition 4.3.3]{kei04'}, \[ \Lip u(x)\le C\limsup_{r\to 0}\frac{1}{r}\dashint_{B(x,r)}|u-u_B|\ud\mu \textrm{ almost everywhere} \] applies to Lipschitz maps with a Banach space target. From this our claim follows by a straightforward application of the Poincar\'e inequality: \[ \limsup_{r\to 0}\frac{1}{r}\dashint_{B(x,r)}|u-u_B|\ud\mu\le C \limsup_{r\to 0}\left(\dashint_{B(x,\sigma r)} g_u^p\ud\mu\right)^{1/p}=Cg_u(x) \] almost everywhere. When the target is a locally compact CAT(0) space $Y$ (which is locally geodesically complete, see section 1.2 below) it is proven in \cite[Corollary 5.10]{oht04} that $\Lip u = g_u$ for locally lipschitz maps $u:X\to Y$.
\end{remark}

\medskip\subsubsection*{$p$-quasicontinuity} It follows from the properties of the $p$-modulus $\Mod_p$ and the definition of $p$-weak upper gradients that, given $u\in \Di pXV$, there exists a curve family $\Gamma$ with $\Mod_p(\Gamma)=0$ so that if $\gamma \notin \Gamma$ then \[ \|u(\gamma(b))-u(\gamma(a))\|_V\le \int_{\gamma|_{[a,b]}}g_u,\ a,b\in [0,1]. \] In particular $u$ is absolutely continuous along $p$-almost every curve $\gamma$.

For us a crucial continuity property is the following concept of $p$-quasicontinuity: a map $u: X\to V$ is said to be $p$-quasicontinuous if for every $\varepsilon >0$ there exists an open set $E\subset X$ with $\Cp_p(E)<\varepsilon$ such that $u|_{X\setminus E}:X\setminus E\to V$ is continuous.

\begin{lemma}
Suppose $(X,d,\mu)$ is a proper metric measure space supporting a weak $(1,p)$-Poincar\'e inequality. Then every map $u\in \Neml pXV$ is $p$-quasicontinuous.
\begin{proof}
Since $p$-quasicontinuity is a local property the claim is implied immediately by \cite[Corollary 6.8]{hei00}.
\end{proof}
\end{lemma}

\begin{remark}
We will often use the following equivalent formulation of $p$-qua\-si\-con\-ti\-nu\-ity: there is a \emph{decreasing} sequence $E_n\supset E_{n+1}$ of open sets in $X$ with $\Cp_p(E_n)<2^{-n}$ such that $u|_{X\setminus E_n}$ is continuous. Indeed, using $p$-quasicontinuity to select open sets $E^k$ with $\Cp_p(E^k)<2^{-k}$ such that $u|_{X\setminus E^k}$ is continuous, the sets $\displaystyle E_n=\bigcup_{k\ge n}E^k$ satisfy the conditions of this alternative formulation.

If $(u_k)_{k\in D}$ is a countable collection of maps in $\Neml pXV$ we may, by a similar procedure, produce a decreasing sequence $E_n\supset E_{n+1}$ of open sets so that $\Cp_p(E_n)<2^{-n}$ and $u_k|_{X\setminus E_n}$ is continuous for all $k\in D, n\in \N$.
\end{remark}

\bigskip\noindent Next we state the Rellich-Kondrakov theorem and another useful result.
\begin{theorem}[Rellich-Kondrakov]\label{rellich}
Let $X$ be a doubling metric measure space supporting a weak $(1,p)-$ Poincar\'e inequality, and $Y$ a proper metric space. If $u_j$ is a sequence in $\Neml pXY$, and $v\in \Neml pXY$ with \[ \sup_{j}\left[\int_Bd_Y(v,u_j)^p\ud\mu+\int_{5\sigma B}g_{u_j}^p\ud\mu\right] <\infty, \] for a given ball $B\subset X$, then there is a subsequence (denoted by the same indices) and $u\in \Nem pBY$ so that $$\|u_j-u\|_{L^p(B;Y)}\to 0$$ as $j\to \infty$ and, moreover, $$\int_{ B}g_u^p\ud\mu\le \liminf_{j\to \infty}\int_{ B}g_{u_j}^p\ud\mu.$$
\end{theorem}
\begin{proof}
Note that the assumptions imply for $q\in Y$,
\begin{align*}
&\sup_j\left[\int_Bd_Y(u_j,q)^p\ud\mu+\int_{5\sigma B}g_{u_j}^p\ud\mu\right] \\
\le &\sup_j\left[2^{p-1}\int_Bd_Y(v,u_j)^p\ud\mu+2^{p-1}\int_Bd_Y(v,q)^p\ud\mu+\int_{5\sigma B}g_{u_j}^p\ud\mu\right]<\infty.
\end{align*}
We have the scalar valued case of the claim by \cite[Theorem 8.3]{haj00}. Using the argument presented in the proof of \cite[Theorem 1.3]{kor93} we may reduce the claim to the scalar valued case, and hence we are done. 
\end{proof}

\begin{lemma}\label{le3}
Suppose $f_n$ is a sequence in $\Di pXV$ and $f_n\to f$ in $L^p_{loc}(X;V)$. If $g_n$ is a sequence of $p$-weak upper gradients of $f_n$ and $g_n\to g$ weakly in $L^p(X)$, then there is $\tilde f\in \Di pXV$ so that $f=\tilde f$ almost everywhere, and $g$ is a ($p$-integrable) $p$-weak upper gradient for $\tilde f$.

Moreover, if $f\in \Di pXY$ then we may choose $\tilde f = f$.
\begin{proof}
By Mazur's lemma \cite[Lemma 6.1]{bjo11} a sequence of convex combinations of the $g_n$'s converge to $g$ in norm. (In particular we may choose the convex combinations so that the $j^{\textrm{th}}$ element is a convex combination of $g_{j},g_{j+1},g_{j+2},\ldots$.) The corresponding sequence of convex combinations of the $f_ n$'s converges to $f$ in $L^p_{loc}(X;V)$ and therefore, by the proof of \cite[Lemma 3.1]{kal01} (see also \cite[Proposition 2.3]{bjo11}) $g$ is a $p$-weak upper gradient for a representative $\tilde f$ of $f$. The $p$-integrability is obvious.

The last assertion follows from the proof of \cite[Proposition 2.3]{bjo11}.
\end{proof}
\end{lemma}

\subsection{Spaces of nonpositive curvature: Busemann and Alexandrov}

\noindent Let us mention to start with that of the two notions of nonpositive curvature, Busemann's and Alexandrov's, the more widely used is the notion given by Alexandrov. However, we shall use Busemann's definition of nonpositive curvature for the simple reason that the nature of the methods used in this paper corresponds quite naturally to the notions used in Busemann's definition.

A central theme in the theory of spaces of nonpositive curvature, both Busemann's and Alexandrov's, is convexity.

\bigskip\noindent Recall that a geodesic $\gamma$ joining two points $x,y\in Y$ in a metric space is a path satisfying $\ell(\gamma)=d(x,y)$. A geodesic $\gamma$ can always be constant speed parametrized so that $\gamma: [0,1]\to Y$ and $d(\gamma(t),\gamma(s))=\ell(\gamma)|t-s|$ for all $t,s\in [0,1]$; see \cite[Proposition 2.2.9]{pap05}.

We call a (path connected) metric space $(Y,d)$ \emph{locally complete and geodesic} if each point has a closed neighbourhood that is a complete geodesic space.


\begin{definition}\label{locconv}
\begin{itemize}
\item[(a)] A metric space $Y$ is called a \emph{Busemann space} if it is locally complete and geodesic, and for every pair of affinely reparametrized geodesics $\gamma,\sigma:[0,1]\to Y$ the distance map \[ t\mapsto d(\gamma(t),\sigma(t)): [0,1]\to \R \] is convex.

\item[(b)] A metric space $Y$ is \emph{locally convex} if each point has a neighbourhood that is a Busemann space with the induced metric. Such neighbourhoods are called Busemann convex neighbourhoods.
\end{itemize}
\end{definition}

Note that many authors define (local) convexity by considering geodesics with common starting point (see, for instance \cite[Chapter II.4]{bri99}). However, this seemingly weaker notion of (local) convexity is easily seen to be equivalent to the definition presented here.

\bigskip\noindent To speak about Alexandrov's notion of nonpositive curvature we need to introduce the concept of geodesic triangles and \emph{comparison} triangles. 

Let $Y$ be a locally complete and geodesic space. A geodesic triangle $\Delta \subset Y$ consists of three points $x,y,z\in Y$ and affinely reparametrized geodesics $\gamma_{xy},\gamma_{xz}, \gamma_{yz}$ connecting $x$ with $y$, $x$ with $z$ and $y$ with $z$, respectively. A comparison triangle $\overline{\Delta}\subset \R^2$ is a Euclidean triangle with vertices $\overline x, \overline y, \overline z$ such that the side lengths agree, i.e. $$|\overline x-\overline y|=d(x,y),\ |\overline x-\overline z|=d(x,z),\ |\overline y-\overline z|= d(y,z).$$ For any geodesic triangle a comparison triangle always exists, \cite[Lemma I.2.14]{bri99}. Given this, the notion of a comparison point to $w\in \Delta$ is self-explanatory.

\begin{definition}
\begin{itemize}
\item[(a)] A complete geodesic space $Y$ is said to be of \emph{global nonpositive curvature} if, for all geodesic triangles $\Delta$ with comparison triangle $\overline\Delta$ and any two points $a,b\in \Delta$, the comparison points $\overline a,\overline b\in \overline{\Delta}$ satisfy \[ d(a,b)\le |\overline a-\overline b|. \]
\item[(b)] A locally complete and geodesic space is said to be of \emph{nonpositive curvature} (an NPC space for short) if each point has a closed neighbourhood that is a space of global nonpositive curvature when equipped with the inherited metric.
\end{itemize}
\end{definition}

\bigskip\noindent In the literature one often encounters the name CAT(0) space for spaces of global nonpositive curvature. 

In intuitive terms a globally nonpositively curved space is one where geodesic triangles are ``thinner'' than their corresponding Euclidean comparison triangles. One sees from the two definitions that Alexandrov's does not directly pertain to convexity whereas Busemann's definition does. It is however true that a nonpositively curved space is locally convex (and similarly for the global notions). The converse fails to hold and so the class of locally convex spaces is strictly larger than that of nonpositively curved ones. In fact a Banach space is of global nonpositive curvature if and only if its norm comes from an inner product. In contrast, a Banach space is a Busemann space if and only if its unit ball is strictly convex. For a good account of convexity in normed spaces see \cite[Chapters 7 and 8]{pap05}, and also \cite[Chapter 4]{bri99}.  The difference between the two notions is that in Alexandrov's definition the points $a,b\in \Delta$ are allowed to be arbitrary while Busemann's definition only allows one to compare certain pairs of points without changing the comparison triangle. (Ones that are of the form $a=\gamma(t)$, $b=\sigma(t)$ for some $\gamma,\sigma\subset \Delta$ and \emph{the same $t$ for both}.)

Let us mention the following result, due to Alexandrov, from which the convexity of a globally nonpositively curved space follows, \cite[Cor. 2.5]{stu03}

\begin{proposition}
Let $Y$ be a space of global nonpositive curvature and let $\gamma,\sigma:[0,1]\to Y$ be two affinely reparametrized geodesics. Then for all $t\in [0,1]$ the inequality 
\begin{eqnarray} 
d^2(\gamma(t),\sigma(t))&\le td^2(\gamma(1),\sigma(1))+(1-t)d^2(\gamma(0),\sigma(0))\nonumber\\ 
&-t(1-t)[d(\gamma(1),\gamma(0))-d(\sigma(1),\sigma(0))]^2
\end{eqnarray}
holds. In particular the metric $d:Y\times Y\to \R$ is convex whence $Y$ is a Busemann space.
\end{proposition}

\begin{corollary}
A nonpositively curved space $Y$ is also locally convex.
\end{corollary}

\bigskip\noindent The main reason for interest in locally convex spaces is the validity of a strong ``local to global'' principle. Below is a very general notion of this, see \cite[The Cartan-Hadamard Theorem 4.1, p.193]{bri99}, but the punchline of the principle is that a simply connected, locally convex metric space is \emph{globally} convex, i.e. a Busemann space (and similarly for the nonpositive curvature case).

\begin{theorem}\label{universal}
Let $Y$ be a locally convex space. Then $Y$ admits a universal covering $\tilde Y$ with a unique metric with the properties that the covering map $\pi:\tilde Y\to Y$ is a local isometry and $\tilde Y$ is a Busemann space.

If $Y$ is of nonpositive curvature (in the sense of Alexandrov) then the universal cover is a CAT(0) space.
\end{theorem}

\noindent In fact the universal covering space may be constructed as follows. Take any $q\in Y$ and consider the set $\tilde Y_q$ consisting of all constant speed parametrized local geodesics $\gamma:[0,1]\to Y$ starting at $q$ (i.e. $\gamma(0)=q$). The map $p_q:\tilde Y_q\to Y$ given by $p_q(\gamma)=\gamma(1)$ is a covering map and the metric of $Y$ pulls back under $p_q$ to produce a length metric $d_q$ on $\tilde Y_q$ such that the claims of Theorem \ref{universal} are valid. For details, see \cite[Chapter II.4]{bri99}.

The proof of Theorem \ref{universal} relies in large part on the following lemma which (along with its refinement \ref{hololemmar}) will be of independent use for us. For a reference see \cite[Lemma 4.3, p. 194]{bri99}.
\begin{lemma} \label{hololemma}
Suppose $Y$ is a locally convex space, $x,y\in X$. Let $x'\in B(x,\varepsilon)$, $y'\in B(y,\varepsilon)$, $\gamma:[0,1]\to Y$ a constant speed parametrized local geodesic joining $x$ and $y$, with $\varepsilon >0$ such that $B(\gamma(t),2\varepsilon)$ is Busemann convex for all $t$. Then there exists a unique constant speed parametrized local geodesic $\alpha:[0,1]\to Y$ joining $x'$ with $y'$ so that 
\begin{equation}\label{eq: hololemma}
t\mapsto d(\gamma(t),\alpha(t)) \textrm{ is a convex function.}
\end{equation}
Moreover, $\alpha$ satisfies \[\ell(\alpha)\le \ell(\gamma)+d(x',x)+d(y',y).\] In particular $d(\gamma(t),\alpha(t))<\varepsilon$ for each $t\in [0,1]$.
\end{lemma}
\noindent
We shall require a sharpening of Lemma \ref{hololemma}, a proof of which can be found in \cite[Theorem 9.2.4]{pap05}.

\begin{lemma}[refinement of Lemma \ref{hololemma}]\label{hololemmar}
Suppose $x,y\in Y$, $\varepsilon >0$ and $\gamma$  are as in the previous lemma; $u,v\in B(x,\varepsilon)$ and $u',v'\in B(y,\varepsilon)$. If $\alpha$ and $\beta$ are  the unique local geodesics provided by Lemma \ref{hololemma}, connecting $u$ with $u'$ and $v$ with $v'$, respectively, then the unique local geodesic connecting $v$ and $v'$ with respect to $\alpha$, provided by Lemma \ref{hololemma} is also $\beta$. 
\end{lemma}

A remark we will use without further mention in the sequel is the following. If $\alpha,\beta$ are two local geodesics with $$\sup_{0\le t\le 1}d_Y(\alpha(t),\beta(t))<\varepsilon$$ and $\varepsilon>0$ is such that $B(\alpha(t),2\varepsilon)$ is Busemann convex for all $t\in [0,1]$, then $t\mapsto d_Y(\alpha(t),\beta(t))$ is convex. This is true since for each $t$ there is a small neighbourhood where $\alpha$ and $\beta$ are both geodesics in the Busemann convex space $B_Y(\alpha(t),2\varepsilon)$ and thus the distance function $d_Y(\alpha,\beta)$ is convex near $t$.

Another particular consequence of Theorem \ref{universal} is a uniqueness property for homotopy classes of paths in locally convex spaces. Here we give a formulation identical to \cite[Corollary 9.3.3]{pap05} apart from the local compactness --assumption made there. For non-locally compact spaces the proof of this proposition is included in the proof of \cite[Corollary 4.7, p. 197]{bri99}.

\begin{proposition}\label{unigeo}
Let $Y$ be a \emph{complete} locally convex space. Then each continuous path $\gamma$ in $Y$ is (endpoint-preserving --) homotopic to a local geodesic $\sigma$, \emph{unique} up to reparametrization.
\end{proposition}

\section{Homotopies}

\noindent Now we introduce a notion of homotopy for $p$-quasicontinuous maps, generalizing the classical one. This so called \emph{$p$-quasihomotopy} is then compared to a different notion, appearing in \cite{bre01, han01, han03}. The notion $p$-quasihomotopy utilizes the geometric structure of the target space whereas the second notion, \emph{path homotopy}, which is stated for maps in the Dirichlet class, in fact relies on the topology of that class. 

\subsection{$p$-quasihomotopy and path homotopy}

Throughout this section $X$ stands for a complete doubling metric measure space $(X,d, \mu)$ supporting a weak $(1,p)$-Poincar\'e inequality for some $p>1$ which will be fixed for the rest of the paper. In the definitions below $Y$ stands for a complete separable metric space. Completeness ensures that $\Di pXY$ is closed and separability ensures that maps in $\Di pXY$ are essentially separably valued. By the remark after Theorem \ref{rmk} it is therefore always enough to find a $p$-integrable ($p$-weak) upper gradient for a map in order to show it is in $\Di pXY$. We shall gradually add assumptions on the target space $Y$.

\begin{definition}\label{phomotopy}
Let $u,v:X\to Y$ be $p$-quasicontinuous. We say that $u$ and $v$ are $p$-quasihomotopic if there exists a map $H\colon X\times [0,1]\to Y$ with the following property.

For every $\varepsilon >0$ there exists an open set $E\subset X$ with $\Cp_p(E)<\varepsilon$ such that $H_{X\setminus E\times [0,1]}$ is a usual homotopy between $u|_{X\setminus E}$ and $v|_{X\setminus E}$.
\end{definition}

The notion of $p$-quasihomotopy is in the spirit of \cite{kor93}, where homotopy of maps into nonpositively curved spaces is studied. Explicit emphasis is given to the topology (structure) of the target space instead of the topology of the Sobolev (or in our case Newtonian) space. 

It is noteworthy that for $p>Q=\log_2 C_\mu$ the $p$-capacity $\Cp_p$ becomes trivial in the sense that $\Cp_p(A)=0$ if and only if $A=\varnothing$. Therefore for these values of $p$, the notions of $p$-quasihomotopy and usual homotopy agree. This is natural in view of the Sobolev embedding theorem \cite[Theorem 6.2]{hei00}, which states that Newtonian maps, for $p>Q$, are in fact $(1-p/Q)$ -- H\"older continuous.

To talk about \emph{path homotopy} one needs to specify the topology used in $\Di pXY$. If not otherwise stated $\Di pXY$ will be equipped with the standard topology.

\begin{definition}\label{path}
We say that $u,v\in \Di pXY$ are \emph{path-homotopic} if there exists a continuous path $h\in C([0,1]; \Di pXY)$ connecting $u$ and $v$. 
\end{definition}
This definition appears in \cite{bre01} where much more concerning it can be found. (See also \cite{bre03, han01, han03} and the references therein.) The study of path homotopy classes is equivalent to the study of path components of the Dirichlet space $\Di pXY$.

\bigskip\noindent We shall see that a path-homotopy satisfying certain rectifiability assumptions can always be modified to become a $p$-quasihomotopy. Conversely, a locally geodesic $p$-quasihomotopy between maps with locally convex target defines a path in the Dirichlet class between the endpoint maps, but continuity must be taken with respect to the Ohta topology. (Such a $p$-quasihomotopy even satisfies some local rectifiability assumptions) 

\bigskip\noindent A first result reflects the geometric structure of the target space $Y$ in the $p$-quasihomotopies between maps. Given a $p$-quasihomotopy $H:u\simeq v$ denote
\begin{align*}
H_t:X\to Y,&\quad H_t(x)=H(x,t)\\
H^x:[0,1]\to Y,&\quad H^x(t)=H(x,t).
\end{align*}

\begin{theorem}\label{geodhom}
Suppose $Y$ is a complete, locally convex space and let $u,v:X\to Y$ be $p$-quasicontinuous and $p$-quasihomotopic. Given a $p$-quasihomotopy $\tilde H:u\simeq v$, there exists a $p$-quasihomotopy $H:u\simeq v$, unique in the following sense. For $p$-quasievery $x\in X$ the path $H^x$ is the unique local geodesic between $u(x)$ and $v(x)$ belonging to the homotopy class of $\tilde H^x$. 
\end{theorem}

\begin{remark}
\noindent
\begin{itemize}
\item[a)] Such a $p$-quasihomotopy is called \emph{locally geodesic}. Sometimes, for brevity, the word ``locally'' is omitted. (This does  not mean that the paths $H^x$ are geodesic.)
\item[b)] We use the notation $H:u\simeq v$ to signify that $H$ is a $p$-quasihomotopy between the maps $u$ and $v$.
\end{itemize}
\end{remark}

\begin{proof}
To prove the claim let $\tilde H:u\simeq v$ be a $p$-quasihomotopy. For $p$ a.e. $x$ set
\[H(x,t)=\gamma^x(t)\]
where $\gamma^x$ is the unique constant speed parametrized local geodesic in the homotopy class of $\alpha^x$, $\alpha^x(t)=\tilde H (x,t)$, given by Proposition \ref{unigeo}.

It suffices to prove that $H$ is a $p$-quasihomotopy. 
To this end let $\varepsilon >0$ be arbitrary and let $E$ be an open set such that $\Cp_p(E)<\varepsilon$ and $\tilde H|_{X\setminus E\times [0,1]}$ is a usual homotopy $u|_{X\setminus E}\simeq v|_{X\setminus E}$. We shall show that $H_{X\setminus E\times [0,1]}$ is also a usual homotopy. (Note in particular that by the choice of $E$, the maps $ H_0|_{X\setminus E}=u|_{X\setminus E}$ and $ H_1|_{X\setminus E}=v|_{X\setminus E}$ are continuous.) If $x\in X\setminus E$ and $\delta>0$ are given, let $\delta_0\le \delta$ be such that  $B(\tilde H_t(x),\delta_0)$ and $B(\gamma^x(t),2\delta_0)$ are Busemann convex, for all $t\in [0,1]$. By the continuity of $\tilde H|_{X\setminus E\times [0,1]}$ we may find $r>0$ such that $\tilde H_t(B(x,r)\setminus E)\subset B(\tilde H_t(x),\delta_0)$ for all $t\in [0,1]$.

For $y\in B(x,r)\setminus E$ let $\gamma'$ be the unique local geodesic with $\gamma'(0)=u(y)$, $\gamma'(1)=v(y)$, guaranteed by Lemmata \ref{hololemma} and \ref{hololemmar} such that 
\[t\mapsto d(\gamma^x(t),\gamma'(t))\]
is convex. 
Then $\gamma'$ is necessarily homotopic to $\alpha^y$: 
\[\gamma'\simeq \beta^{yx}_u\cdot\gamma^x\cdot \beta^{xy}_v \simeq \beta^{yx}_u\cdot\alpha^x\cdot \beta^{xy}_v \simeq \alpha^y.\]
Here $\beta^{xy}_u$ is the geodesic from $u(x)$ to $u(y)$ and $\beta^{yx}_v$ the geodesic from $v(y)$ to $v(x)$. The last homotopy follows since for all $t$ the points $\alpha^y(t)$ belong to the Busemann convex ball $ B(\alpha^x(t),\delta_0)$ by the choices of $y$ and $\delta$.

This shows that in fact $\gamma'=\gamma^y$ (by uniqueness) and from the estimates in Lemma \ref{hololemma} we have, for $y,z\in B(x,r)\setminus E$ 
\begin{align*}
d_Y(H_t(y),H_t(z))\le &\ td_Y(v(y),v(z))+(1-t)d_Y(u(y),u(z)) \\
&\textrm{and} \\
d_Y(H_t(z),H_s(z))\le &\ |t-s|\ell(\gamma^z)\le |t-s|(\ell(\gamma^x)+d_Y(u(x),u(z))+d_Y(v(x),v(z))).
\end{align*}
These estimates prove the continuity of $H|_{X\setminus E\times [0,1]}$.
\end{proof}

\begin{theorem}\label{convexi}
Suppose $Y$ is a separable, complete locally convex space. If $H:u\simeq v$ is a locally geodesic $p$-quasihomotopy between two $p$-quasicontinuous maps $u,v: X\to Y$ then it satisfies the following convexity estimate: whenever $x\in X$ and $\varepsilon >0$ is such that $B(H_t(x),2\varepsilon)$ is a convex ball for all $t\in [0,1]$, $y,z\in X$ satisfy $\displaystyle \max_{0\le t\le 1}d_Y(H_t(x),H_t(y))<\varepsilon$ and $\displaystyle \max_{0\le t\le 1}d_Y(H_t(x),H_t(z))<\varepsilon$, we have
\begin{eqnarray}
d_Y(H_t(y),H_t(z)) \le & td_Y(v(y),v(z))+(1-t)d_Y(u(y),u(z)) \label{eq: convex}\\
|\ell(H^y)-\ell(H^z)| \le & d_Y(u(y),u(z))+d_Y(v(y),v(z)). \label{eq: est}
\end{eqnarray}
Here $H^w$ denotes the local geodesic $t\mapsto H_t(w)$, $w\in X$.
\begin{proof}
The paths $\gamma_1 =t\mapsto H_t(y)$ and $\gamma_2 =t\mapsto H_t(z)$ are local geodesics. For each $t_0\in [0,1]$ there is a neighbourhood $U\ni t_0$ such that $\gamma_1|_U$ and $\gamma_2|_U$ are geodesics in the Busemann space $B(H_{t_0}(x),2\varepsilon)$ and thus the function $d_Y(H_t(y),H_t(z))=d_Y(\gamma_1(t),\gamma_2(t))$ is convex in $U$. Therefore it is convex in $[0,1]$, proving (\ref{eq: convex}).

To prove (\ref{eq: est}) we may assume, without loss of generality, that $\ell(H^y)\ge \ell(H^z)$. Let $\gamma'$ be the local geodesic from $u(y)$ to $v(z)$ guaranteed by Lemmata \ref{hololemma} and \ref{hololemmar}. Then by the convexity of $d_Y(H^y,\gamma')$ and the local geodesic property we have, for small $t>0$ 
\begin{align*}
t\ell(H^y)=& d_Y(u(y),H_t(y))=d_Y(\gamma'(0),H_t(y))\le d_Y(\gamma'(0),\gamma'(t))+d_Y(\gamma'(t),H_t(y)) \\
\le & t\ell(\gamma')+td_Y(v(z),v(y)).
\end{align*}
The same argument for the inverse paths $(\gamma')^{-1}$ and $(H^z)^{-1}$ yields \[ t\ell((\gamma')^{-1}) \le t\ell((H^z)^{-1})+td_Y(u(z),u(y)). \]
Cancelling out $t$ and moving $\ell(H^z)=\ell((H^z)^{-1})$ to the other side we obtain (\ref{eq: est}).
\end{proof}
\end{theorem}

\subsection{$p$-quasihomotopies as paths in the Dirichlet class}

\noindent We may view a $p$-quasihomotopy $H:u\simeq v$ as gliding the map $u$ to $v$ through the path $t\mapsto H_t$ in a quasicontinuous manner. Our aim in this subsection is to develop this view and study $p$-quasihomotopies as paths in the Dirichlet class $\Di pXY$. To pass from pointwise information to paths in the Dirichlet class we need sufficiently good geometric behaviour from the target space. We continue assuming that $(X,d,\mu)$ is a complete doubling metric measure space supporting a weak $(1,p)$-Poincar\'e inequality and $Y$ is a (complete and separable) locally convex space (see Subsection 1.4).


The following stronger version of \cite[Proposition 1.48]{bjo11} will prove very useful for us.

\begin{lemma}\label{refinement}
Let $E_n\subset X$ be a sequence of sets with $\varepsilon_n:=\Cp_p(E_n)$ converging to zero. Denote \[ \Gamma_\infty = \{ \gamma: \gamma^{-1}(E_n)\ne \varnothing \ \forall n \}. \] Then $\Mod_p(\Gamma_\infty)=0$.
\begin{proof}
Let $u_m$ be such that $u_m|_{E_m}=1$, $u_m\ge 0$ and $g_m$ an upper gradient of $u_m$ with $$ \int_X(|u_m|^p+g_m^p)\le 2\varepsilon_m.$$ Since $\varepsilon_m\to 0$ we have $u_m\to 0$ in $\Ne pX$ as $m\to \infty$. Consequently we may pass to a subsequence $u_{m_k}$ converging to zero outside a set $F$ of $p$-capacity zero and satisfying \[ \sum_{k=1}^\infty \varepsilon_{m_k}^{1/p}<\infty. \].

For $l \ge 1$, set \[ \rho_l=\sum_{k\ge l}g_{m_k}. \] Then \[ \left(\int_X\rho_l^p\ud\mu\right)^{1/p}\le \sum_{k\ge l}\left(\int_Xg_{m_k}^p\ud\mu\right)^{1/p}\le 2\sum_{k\ge l}\varepsilon_{m_k}^{1/p}. \] Denote by $\Gamma_F$ the family of paths $\gamma$ with the property that $\gamma^{-1}(F)\ne \varnothing$. Then for $\gamma \in \Gamma_\infty\setminus\Gamma_F$ we have that for each $k$ there exists $t_k\in [0,1]$ with $\gamma(t_k)\in E_{m_k}$ while $\gamma(0)\notin F$. Given $l\ge 1$ we have, for all $k\ge l$ the estimate \[ |u_{m_k}(\gamma(0))-1|=|u_{m_k}(\gamma(0))-u_{m_k}(\gamma(t_k))|\le \int_\gamma g_{m_k}\le \int_\gamma \rho_l. \] Taking $k\to\infty$ we have $u_{m_k}(\gamma(0))\to 0$, whence $$ \int_\gamma \rho_l\ge 1.$$ This shows that \[ \Mod_p(\Gamma_\infty\setminus \Gamma_F)\le \int_X\rho_l^p\ud\mu \le 2^p(\sum_{k\ge l}\varepsilon_{m_k}^{1/p})^p\to 0 \] as $l\to \infty$. Since by \cite[Proposition 1.48]{bjo11} we have $\Mod_p(\Gamma_F)=0$ it follows that $$ \Mod_p(\Gamma_\infty)\le \Mod_p(\Gamma_\infty\setminus\Gamma_F)+\Mod_p(\Gamma_F)=0$$ and the proof is complete.
\end{proof}
\end{lemma}

\begin{theorem}\label{di}
Let $u,v\in \Di pXY$ and $H:u\simeq v$ be a locally geodesic $p$-quasihomotopy. Then for $t\in [0,1]$ we have \[ g_{H_t}\le tg_v+(1-t)g_u \] almost everywhere. In particular $H_t\in \Di pXY$ for all $t$.
\end{theorem}
\begin{proof}
Let $E_m\supset E_{m+1}$ be a sequence of open sets in $X$, with $\Cp_p(E_m)<2^{-m}$ and $H|_{X\setminus E_m\times [0,1]}$ continuous. By Lemma \ref{refinement} and the fact that $u$ and $v$ are absolutely continuous on $p$-almost every curve, there is a curve family $\Gamma$ with $\Mod_p(\Gamma)=0$ such that each $\gamma\notin \Gamma$ satisfies
\begin{enumerate}
\item there exists $m_0$ so that $\gamma^{-1}(E_{m_0})=\varnothing$ (and consequently $\gamma^{-1}(E_m)=\varnothing$ for every $m\ge m_0$ )
\item the inequalities
\begin{align*}
d_Y(u(\gamma(b)),u(\gamma(a)))\le \int_{\gamma|_{[a,b]}}g_u \\
d_Y(v(\gamma(b)),v(\gamma(a)))\le \int_{\gamma|_{[a,b]}}g_v
\end{align*}
hold for $a,b\in [0,1]$.
\end{enumerate}
Fix such a $\gamma$ and let $K=|\gamma|\subset X\setminus E_{m_0}$. Since $H|_{K\times [0,1]}$ is uniformly continuous there exists $\varepsilon >0$ so that $B(H_t(\gamma(s)),2\varepsilon)$ is convex for all $t,s\in [0,1]$. Futhermore the uniform continuity implies the existence of $\delta>0$ so that $$\displaystyle \max_{0\le t\le 1}d_Y(H_t(\gamma(b)),H_t(\gamma(a)))<\varepsilon$$ whenever $|a-b|<\delta$. By the estimate (\ref{eq: convex}) we therefore have 
\begin{align*}
d_Y(H_t(\gamma(b)),H_t(\gamma(a)))\le & td_Y(v(\gamma(b)),v(\gamma(a)))+(1-t)d_Y(u(\gamma(b)),u(\gamma(a))) \\
\le & \int_{\gamma|_{[a,b]}}(tg_v+(1-t)g_u).
\end{align*}
Partitioning $[0,1]$ into subintervals of length $<\delta$ and applying the estimate above yields \[d_Y(H_t(\gamma(1)),H_t(\gamma(0)))\le \sum_{k}d_Y(H_t(\gamma(a_k)),H_t(\gamma(a_{k-1})))\le \int_\gamma (tg_v+(1-t)g_u). \]
This proves that $tg_v+(1-t)g_u$ is a $p$-weak upper gradient for $H_t$, and the claim follows.
\end{proof}

\begin{lemma}\label{locint}
Suppose $u,v\in \Di pXY$ and $H:u\simeq v$ is a locally geodesic $p$-quasihomotopy. Define $l_H:X\to \R$ by $$l_H(x)=\ell(H^x).$$ Then $l_H\in \Dir pX$, and \[ g_{l_H}\le g_u+g_v. \]
\end{lemma}
\begin{proof}
It suffices to prove the inequality. For this let $\Gamma$ be as in the previous proof and fix $\gamma\notin \Gamma$. By the reasoning in the previous proof we have the existence of $\delta>0$ so that $|l_H(\gamma(b))-l_H(\gamma(a))|\le d_Y(u(\gamma(b)),u(\gamma(a)))+ d_Y(v(\gamma(b)),v(\gamma(a)))$ (estimate (\ref{eq: est})) whenever $|a-b|<\delta$, $a,b\in [0,1]$. By the same partitioning argument we arrive at \[ |l_H(\gamma(1))-l_H(\gamma(0))|\le \int_\gamma(g_u+g_v) \] and this proves the claim.
\end{proof}
\begin{corollary}\label{convloc}
In the situation of Lemma \ref{locint} we have, for each compact $K\subset X$, the inequality \[ \int_Kd_Y^p(H_t,H_s)\ud\mu \le |t-s|^p\int_Kl_H^p\ud\mu. \] Consequently $H_s\to H_t$ in $L^p_{loc}(X;Y)$ as $s\to t$.
\begin{proof}
The inequality follows directly from the fact that $$d_Y(H_t(x),H_s(x))\le |t-s|\ell(H^x)=|t-s|l_H(x)$$ for $p$-quasievery $x\in X$. The second claim is immediate from the first and the fact that $l_H\in L^p_{loc}(X)$.
\end{proof}
\end{corollary}

\begin{theorem}\label{polku}
Suppose $H:u\simeq v$ is a locally geodesic $p$-quasihomotopy. Then the map $\alpha:[0,1]\to \Di pXY$, given by  $$\alpha(t)= H_t,$$ is a continuous path when $\Di pXY$ is equipped with the Ohta topology.
\end{theorem}
\begin{proof}
From Theorem \ref{di} we have that $t\mapsto H_t$ is a map $[0,1]\to \Di pXY$. That $H_s\to H_t$ in $L^1_{loc}(X;Y)$ as $s\to t$ follows from Corollary \ref{convloc}. Therefore we only need to focus on the convergence of the $p$-weak upper gradients. Let $t\in [0,1]$. We will show that the one-sided limits exist and agree: \[ \lim_{s\to t^+}g_{H_s}=g_{H_t}=\lim_{s\to t^-}g_{H_s} \] in the $L^p$-sense. We will make use of the following well known fact about uniformly convex Banach spaces:
\begin{fact}
Let $V$ be a uniformly convex Banach space, $x_k$ converges weakly to $x$ as $k\to\infty$ and further $\|x_k\|\to \|x\|$. Then $x_k\to x$ in norm.
\end{fact}
In fact it suffices to prove that $g_{H_s}\to g_{u}$ as $s\to 0$. This is because of the following: the restriction $H|_{X\times [t,1]}$ is a $p$-quasihomotopy between $H_t$ and $v$, so by rescaling the parameter side we obtain a locally geodesic $p$-quasihomotopy $\tilde H: H_t\simeq v$, $\tilde H(x,s)=H(x,t+s(1-t))$ so that $$ \lim_{s\to t^+}g_{H_s}=\lim_{s\to 0}g_{\tilde H_s};$$ to study $\lim_{s\to t^-}g_{H_s}$ we simply replace $\tilde H$ by $\hat H: H_t\simeq u$, $\hat H(x,s)=H(x,t(1-s))$.

\bigskip\noindent Now, to prove that $\lim_{s\to 0}g_{H_s}= g_u$ in $L^p(X)$, take any sequence $s_k\to 0$. Since 
\begin{align*}
\|g_{H_{s_k}}\|_{L^p(X)}\le s_k\|g_v\|_{L^p(X)}+(1-s_k)\|g_u\|_{L^p(X)}\le \|g_v\|_{L^p(X)}+\|g_u\|_{L^p(X)}
\end{align*}
for all $k$, the reflexivity of $L^p(X)$ implies that there is a subsequence (denoted by the same indices) converging to some $g\in L^p(X)$, whence \[ \|g\|_{L^p(X)}\le \liminf_{k\to \infty}\|g_{H_{s_k}}\|_{L^p(X)}. \] On the other hand, since $H_{s_k}\to u$ in $L^p_{loc}(X;Y)$ (Corollary \ref{convloc}) it follows by Lemma \ref{le3} that $g$ is a $p$-weak upper gradient for $u$. This and the convexity estimate together imply 
\begin{align}\label{conestimate}
\limsup_{k\to \infty}\|g_{H_{s_k}}\|_{L^p(X)}&\le \limsup_{k\to\infty}[s_k\|g_v\|_{L^p(X)}+(1-s_k)\|g_u\|_{L^p(X)}]\nonumber \\
&\le \|g_u\|_{L^p(X)}\le \|g\|_{L^p(X)},
\end{align}
where the last inequality comes from the minimality of $g_u$. Consequently $$\|g_{H_{s_k}}\|_{L^p(X)}\to \|g\|_{L^p(X)},$$ therefore $g_{H_{s_k}}\to g$ in norm. 

Let us still prove that $g= g_u$. From the fact that $g$ is a $p$-weak upper gradient for $u$  it follows that $g_u\le g$ almost everywhere, so it suffices to prove $\|g\|_{L^p(X)}\le \|g_u\|_{L^p(X)}$. This, however follows immediately the above estimate (\ref{conestimate}):
\begin{align*}
\|g\|_{L^p(X)}= \lim_{k\to \infty}\|g_{H_{s_k}}\|_{L^p(X)}\le \limsup_{k\to \infty}\|g_{H_{s_k}}\|_{L^p(X)}\le \|g_u\|_{L^p(X)}.
\end{align*}

\bigskip\noindent Altogether we have shown that for every $s_k\to 0$ we have $g_{H_{s_k}}\to g_u$ in $L^p(X)$ up to a subsequence. Therefore $g_{H_{s_k}}\to g_u$ in $L^p(X)$ and the proof is complete.
\end{proof}

\begin{remark}\label{rem}
Theorems \ref{geodhom} and \ref{polku} together imply that $p$-quasihomotopic maps into locally convex targets are also path homotopic if $\Di pXY$ is equipped with the Ohta topology.

To see this suppose $u,v\in \Di pXY$ are $p$-quasihomotopic. By Theorem \ref{geodhom} they may be connected by a locally geodesic $p$-quasihomotopy. Theorem \ref{polku} then implies that $u$ and $v$ are path-homotopic in the Ohta topology.
\end{remark}

\subsection{Pointwise properties of path homotopies.}

\noindent In this subsection we prove Theorem \ref{reimpliesp}. Before that let us make some remaks.

\subsubsection*{Remarks} 
\begin{enumerate}
\item We pose no control on the $L^p$-norms of $d_Y(h_t,h_s)$.
\item The conditions of the theorem are a sort of rectifiability requirement for the path $h$. Since $g_{d_Y(u,v)}\le g_{u-v}$ almost everywhere the condition is implied if $h$ is a rectifiable path $h:[0,1]\to \Di pXY$ when $\Di pXY$ is equipped with the standard topology. 
\item Note that $g_{d_Y(u,v)}$ may vanish without the same being true of $g_{u-v}$ (think of maps $u\equiv 0\in \R^n$ and $v$ taking values in $S^{n-1}$).
\item The relation between $g_{d_Y(u,v)}$ and $|g_u-g_v|$ is not clear; when $Y=\R$ we have $g_{d_Y(u,v)}=g_{|u-v|}=g_{u-v}\ge |g_u-g_v|$ but the previous example shows that $g_{d_Y(u,v)}$ may vanish without $|g_u-g_v|$ vanishing. This question is related to the open question \cite[2.13]{bjo11}.
\item The proof below yields a slightly stronger claim that Theorem \ref{geodhom}: a path satisfying the conditions may be modified by possibly changing its values in a negligible set (each $[0,1]$-slice of which has measure zero) so that it becomes a $p$-quasihomotopy.
\end{enumerate}

\begin{proof}[Proof of Theorem \ref{reimpliesp}]
Let $D_n=\{ k/2^n: k=0,\ldots 2^n \}$ and $\displaystyle D=\bigcup_nD_n$ (the dyadic rationals on the interval $[0,1]$). We may find a sequence $E_m\supset E_{m+1}$ of open subsets of $X$ with $\Cp_p(E_m)<2^{-m}$ and $h_s|_{X\setminus E_m}$ continuous for all $m\in \N$ and $s\in D$. For $\displaystyle x\notin E:=\bigcap_mE_m$ define
\begin{align*}
L_n(x)&=2^{n(1-1/p)}\left(\sum_{k=1}^{2^n} d_Y^p(h_{k/2^n}(x),h_{(k-1)/2^n}(x))\right)^{1/p},\textrm{ and}\\
L(x)&=\sup_nL_n(x).
\end{align*}
Note that $L_n$ is pointwise increasing: denoting $$d_{j,n}=d_{j,n}(\cdot)=d_Y(h_{j/2^n}(\cdot),h_{(j-1)/2^n}(\cdot))$$ we have $d_{j,n}\le d_{2j,n+1}+d_{2j-1,n+1}$, so that
\begin{align*}
L_n&=2^{n(1-1/p)}\left(\sum_{j=1}^{2^n}d_{j,n}^p\right)^{1/p}\le 2^{n(1-1/p)}\left(\sum_{j=1}^{2^n}2^{p-1}(d_{2j,n+1}^p+d_{2j-1,n+1}^p)\right)^{1/p} \\
&= 2^{(n+1)(1-1/p)}\left(\sum_{j=1}^{2^{n+1}} d_{j,n+1}^p\right)^{1/p}=L_{n+1}.
\end{align*}
We shall use this notation throughout the proof of Theorem \ref{reimpliesp}. (If we set $p=1$ in the definition of $L_n$, we are in fact measuring the ``length'' of the ``path'' $D\ni s\mapsto h_s(x)$. The finiteness of this ``length'' however only guarantees that $s\mapsto h_s(x)$ is a sort of BV-map.) The significance of $L$ (as defined above) is shown by the next lemma.

\begin{lemma}\label{le1}
If $L(x)<\infty$ then the map $D\ni s\mapsto h_s(x)$ extends to a $(1-1/p)$-H\"older continuous path, denoted $h^x:[0,1]\to Y$ (joining the points $u(x)$ and $v(x)$) with H\"older constant $L(x)$.
\begin{proof}[Proof of Lemma \ref{le1}]
Let us define $h^x(s)$ for $s\in D$ by $h^x(s)=h_s(x)$. For $n\in\N$ and $j=1,\ldots, 2^n$ we have \[ d_Y^p(h^x(j/2^n),h^x((j-1)/2^n))\le 2^{-n(p-1)}L_n(x)^p\le 2^{-n(p-1)}L(x)^p.\] For $k,l\in \{0,\ldots , 2^n\}$, $l>k$, the triangle inequality implies \[ d_Y(h^x(l/2^n),h^x(k/2^n))\le \sum_{j=k+1}^l d_Y(h^x(j/2^n),h^x((j-1)/2^n))= \sum_{k<j\le l}d_{j,n}(x).\] We may use the H\"older inequality as follows: 
\begin{align*}
\sum_{k<j\le l} d_{j,n}(x)=&\sum_{k<j\le l}1\cdot d_{j,n}(x) \le \left(\sum_{k<j\le l}1^{p/(p-1)}\right)^{1-1/p}\left(\sum_{k<j\le l}d_{j,n}(x)^{p}\right)^{1/p}\\
=& (l-k)^{1-1/p}2^{-n(1-1/p)}2^{n(1-1/p)}\left(\sum_{k<j\le l}d_{j,n}(x)^{p}\right)^{1/p}\\
\le & \left(\frac{l-k}{2^n}\right)^{1-1/p}L_n(x).
\end{align*}
If $s,t\in D$ we may write $s=k/2^n$ and $t= l'/2^m$. Assuming, without loss of generality that $n\ge m$ we have $t=2^{n-m}l'/2^n, s=k/2^n$ and putting the above estimates together yields \[ d_Y(h^x(t),h^x(s))\le |t-s|^{1-1/p}L_n(x)\le |t-s|^{1-1/p}L(x).\] This proves the claim.
\end{proof}
\end{lemma}

\noindent The result is very much in the spirit of the Sobolev embeddings; by this analogy the need for $p>1$ in the definition of $L_n$ becomes apparent.

\bigskip\noindent The rest of the proof is devoted to obtaining pointwise control over $L$. We start with the following lemma.
\begin{lemma}\label{le2}
We have $L\in L^p_{loc}(X)$ and $L_n\to L$ pointwise everywhere. Furthermore the functions \[ g_n:=2^{n(1-1/p)}\left(\sum_{k=1}^{2^n}g_{d_Y(h_{k/2^n},h_{(k-1)/2^n})}^p\right)^{1/p}\]
are $p$-weak upper gradients for $L_n$, satisfying \[ \sup_n\|g_n\|_{L^p(X)}\le C\] where $C$ is the constant in the assumption of the Theorem.
\begin{proof}[Proof of Lemma \ref{le2}]
Since $(L_n)_n$ is a pointwise increasing sequence convergence everywhere follows. By the monotone convergence theorem \[ \int_KL^p\ud\mu=\lim_{n\to \infty}\int_KL_n^p\ud\mu, \quad K\subset X\textrm{ compact.}\] The Poincar\'e inequality for balls $B\subset X$ together with assumption of the Theorem implies that 
\begin{align*}
\left(\dashint_Bd_Y^p(h_t,h_s)\ud\mu\right)^{1/p}& \le \dashint_Bd_Y(h_t,h_s)\ud\mu+Cr\left(\dashint_{\sigma B} g^p_{d_Y(h_t,h_s)}\ud\mu\right)^{1/p} \\
&\le C_B/\mu(B)|t-s|+Cr/\mu(B)^{1/p}|t-s|=C'_B |t-s|.
\end{align*}
Applying this to $d_{k,n}:=d_Y(h_{k/2^n},h_{(k-1)/2^n})$ we have
\begin{align*}
\int_BL_ n^p\ud\mu =& \int_B2^{n(p-1)}\sum_{k=1}^{2^n}d_{k,n}^p\ud\mu\\
=&2^{n(p-1)} \sum_{k=1}^{2^n}\int_Bd_{k,n}^p\ud\mu \le 2^{n(p-1)}\sum_{k=1}^{2^n}\mu(B)(C'_B)^p2^{-np}\\
=&\mu(B)(C'_B)^p
\end{align*}
for all $n\in \N$ and consequently $L\in L^p_{loc}(X)$. (Incidentally, this implies that $L(x)<\infty$ almost everywhere.)

\bigskip\noindent For the second claim fix a family of curves $\Gamma$ with $\Mod_p(\Gamma)=0$ so that whenever $\gamma\notin\Gamma$ the upper gradient inequality \[ |d_{k,n}(x)-d_{k,n}(y)|\le \int_\gamma g_{d_{k,n}}\] is satisfied. For these curves we may estimate
\begin{align*}
&|L_n(x)-L_n(y)|= \\ 
& 2^{n(1-1/p)}\left|\left(\sum_{k=1}^{2^n}d_{k,n}^p(x)\right)^{1/p}-\left(\sum_{k=1}^{2^n}d_{k,n}^p(y)\right)^{1/p}\right|\\
&\le 2^{n(1-1/p)}\left(\sum_{k=1}^{2^n}|d_{k,n}(x)-d_{k,n}(y)|^p\right)^{1/p}\\
&\le 2^{n(1-1/p)}\left(\sum_{k=1}^{2^n}\left(\int_\gamma g_{d_{k,n}}\right)^p\right)^{1/p}.
\end{align*}
The rightmost term may be estimated using the Minkowski inequality in integral form \cite[Theorem 202, p. 148]{har88} by \[ 2^{n(1-1/p)}\int_\gamma \left(\sum_{k=1}^{2^n}g_{d_{k,n}}^p\right)^{1/p}.\] We arrive at \[ |L_n(x)-L_n(y)|\le \int_\gamma g_n.\]

\bigskip\noindent To see the last part use the condition in the statement of the theorem to compute
\begin{align*}
\int_Xg_ n^p\ud\mu=2^{n(p-1)}\sum_{k=1}^{2^n}\int_Xg_{d_{k,n}}^p\ud\mu\le 2^{n(p-1)}\sum_{k=1}^{2^n}C^p2^{-np}=C^p.
\end{align*}
This completes the proof of Lemma \ref{le2}.
\end{proof}
\end{lemma}

\bigskip\noindent Since $(g_n)$ is bounded in $L^p(X)$ there is a subsequence converging weakly to some $g\in L^p(X)$. By Mazur's lemma a sequence of  convex combination of $g_n$'s converges to $g$ in $L^p$. The corresponding sequence of convex combination of $L_n$'s converges everywhere to $L$ and $L<\infty$ almost everywhere. Therefore by \cite[Proposition 2.4]{bjo11} $g$ is a $p$-integrable $p$-weak upper gradient of $L$.

\bigskip\noindent We conclude that $L\in \Dir pX$. In particular $L$ is $p$-quasicontinuous and finite $p$-quasieverywhere. 

\bigskip\noindent Define $H(x,t)=h^x(t)$ for every $x\in X$ for which $L(x)<\infty$, $h^x$ being the path from $u(x)$ to $v(x)$ given by Lemma \ref{le1}. Let us prove that $H$ is a $p$-quasihomotopy.

To this end let $F_m\supset F_{m+1}$ be a sequence of open sets in $X$  with $\Cp_p(F_m)<2^{-m}$ and $L|_{X\setminus F_m}$ continuous, for $m\in \N$. Set $U_m= E_m\cup F_m$. We claim that $H|_{X\setminus U_m\times [0,1]}$ is a continuous homotopy between $u|_{X\setminus U_m}$ and $v|_{X\setminus U_m}$, for all $m$.

It is clear that $H_0=u$ and $H_1=v$ $p$-quasieverywhere so only the continuity remains to be proven. Let $x_k\in X\setminus U_m$, $t_k\in [0,1]$, $(x_k,t_k)\to (x,t)$ where $x\in X\setminus U_m$. There is a compact set $K\subset X\setminus U_m$ containing all $x_k$'s, and $\displaystyle \sup_{z\in K}L(z)<\infty$. Therefore the paths $h^{x_k}$ are equicontinuous and pointwise bounded (since $h^{x_k}(s)=h_s(x_k)\to h_s(x)=h^x(s), s\in D$). By the Arzela-Ascoli theorem $h^{x_k}$ converges uniformly up to a subsequence to a path $\gamma$. But since $h^{x_k}\to h^x$ pointwise in a dense set $D$ it follows that $\gamma= h^x$. This argument shows that any subsequence of $h^{x_k}$ has a further subsequence converging uniformly to $h^x$. From this it follows that $h^{x_k}\to h^x$ uniformly. In particular $h^{x_k}(t_k)\to h^x(t)$, as $k\to \infty$. The proof of Theorem \ref{reimpliesp} is now complete.
\end{proof}

\begin{corollary}
Suppose $Y$ is a locally convex space and $H:u\simeq v$ a locally geodesic $p$-quasihomotopy. Then two maps $u,v\in \Di pXY$ are $p$-quasihomotopic if and only if there exists a path joining $u$ and $v$, satisfying the conditions of Theorem \ref{reimpliesp}.
\begin{proof}
It is not difficult to see, using the argument in the proofs of Theorem \ref{di} and Lemma \ref{locint} that $g_{d_Y(H_t,H_s)}\le |t-s|g_{l_H}$. This, together with Corollary \ref{convloc}, implies that a locally geodesic $p$-quasihomotopy satisfies the conditions of Theorem \ref{reimpliesp}.
\end{proof}
\end{corollary}
\noindent The section is closed by the proof of Corollary \ref{rect}.
\begin{proof}[Proof of Corollary \ref{rect}]
Suppose $h:[0,1]\to \Nem pXY$ is a continuous rectifiable path joining $u,v\in \Nem pXY$. Denote by $\tilde h:[0,1]\to \Nem pXY$ the constant speed parametrization of $h$. The path $\tilde h$ is Lipschitz in the standard metric of $\Nem pXY$, i.e. $$\left(\int_Xd_Y(\tilde h_t,\tilde h_s)^p\ud\mu\right)^{1/p}+\left(\int_Xg_{\tilde h_t-\tilde h_s}^p\ud\mu\right)^{1/p}\le C|t-s|, \quad t,s\in [0,1]. $$ By the inequality
\begin{align*}
g_{d_Y(\tilde h_t,\tilde h_s)}\le g_{\tilde h_t-\tilde h_s}\textrm{ a.e.}
\end{align*}
for any $t,s\in [0,1]$ we see that $\tilde h$ satisfies the assumptions of Theorem \ref{reimpliesp}. The claim follows from this.
\end{proof}

\section{''Lifting'' $p$-quasihomotopies}

\noindent Besides thinking of $p$-quasihomotopies as paths in $\Di pXY$, there is another way of looking at them. In this section we concentrate on this view, which is reminiscent of lifting paths in covering space theory. 

The aim is to view a (locally geodesic) $p$-quasihomotopy $H:u\simeq v$ between two maps $u,v\in \Di pXY$ as a single Newtonian map, with target space $\diaco Y$ a certain covering space of $Y\times Y$.

\subsection{The diagonal cover} We start by constructing the diagonal covering space $\diaco Y$ and recalling some useful facts. Throughout this section $(X,d,\mu)$ stands for a complete doubling metric measure space supporting a weak $(1,p)$-Poincar\'e inequality.

\bigskip\noindent Let $Y$ be a locally convex space. Equip $Y^2$ with the metric \[ d_{Y^2}^2((x_1,y_2);(x_2,y_2))= d_Y^2(x_1,x_2)+d_Y^2(y_1,y_2).\] The product space $Y^2$ remains a locally convex space -- and nonpositively curved in case $Y$ is nonpositively curved. Set \[ \diaco Y= \{\gamma:[0,1]\to Y:\gamma\textrm{ a constant speed local geodesic} \}. \] With metric $\displaystyle d_\infty(\alpha,\beta)=\max_{0\le t\le 1}d_Y(\alpha(t),\beta(t))$ the map \[ \phi:\diaco Y\to Y\times Y, \ p(\gamma)=(\gamma(0),\gamma(1)) \] is a local bilipschitz map.

Suppose $\alpha,\beta$ are two local geodesics with $d_\infty(\alpha,\beta)<\varepsilon$ where $\varepsilon>0$ is such that $B_Y(\alpha(t),2\varepsilon)$ is Busemann convex for all $t\in [0,1]$. Then $t\mapsto d_Y(\alpha(t),\beta(t))$ is convex following the remark after Lemma \ref{hololemmar}. (In particular if $\alpha$ and $\beta$ agree at $0$ and $1$ and $d_\infty(\alpha,\beta)<\varepsilon$ the convexity of the distance function implies $\alpha=\beta$.)

This implies $$ d_\infty(\alpha,\beta)\le \max\{d_Y(\alpha(0),\beta(0)); d_Y(\alpha(1),\beta(1))\}\le  d_{Y^2}(p(\alpha),p(\beta))$$ while the estimate $$d_{Y^2}(p(\alpha),p(\beta))\le \sqrt 2d_\infty(\alpha,\beta)$$ holds always. Therefore $p$ restricted to $B_\infty(\gamma, \varepsilon)$ is a $\sqrt 2$-bilipschitz map $B_\infty(\gamma, \varepsilon)\to p(B_\infty(\gamma, \varepsilon))$.

\bigskip\noindent We may pull back the length metric from $Y\times Y$ to obtain a unique length metric $d_{\widehat Y}$ on $\diaco Y$ such that $\phi:(\diaco Y, d_{\widehat Y})\to (Y\times Y,d_{Y^2})$ is a local isometry. (This metric is given by $\displaystyle d_{\widehat Y}(\alpha,\beta):=\inf_h \ell(p\circ h)$ where the infimum is taken over all the paths $h$ in $\diaco Y$ joining $\alpha$ and $\beta$.)

In particular we have the following Lemma, which is a kind of quantitative version of the local isometry of $\phi$.

\begin{lemma}
Let $\sigma\in \diaco Y$ and $\varepsilon>0$ be such that $B_Y(\sigma(t),2\varepsilon)$ is convex for all $t\in [0,1]$. Then $\phi:B_{\diaco Y}(\sigma,\varepsilon)\to B_{Y^2}(\phi(\sigma),\varepsilon)$ is a surjective isometry.
\begin{proof}
First note that $B_{Y^2}(\phi(\sigma),\varepsilon)\subset B_Y(\sigma(0),\varepsilon)\times B_Y(\sigma(1),\varepsilon)$ is a Busemann convex neighbourhood of $\phi(\sigma)$. By the general theory $\phi$ is a $1$-Lipschitz map: $$d_{Y^2}(\phi(\alpha),\phi(\beta))\le d_{\widehat Y}(\alpha,\beta).$$ Next let us show that if $\gamma\in B_{\diaco Y}(\sigma,\varepsilon)$ then $t\mapsto d_Y(\sigma(t),\gamma(t))$ is convex.

Suppose $d_{\widehat Y}(\sigma,\gamma)<\varepsilon$ and take a path $h$ in $B_{\diaco Y}(\sigma,\varepsilon)$ joining $\sigma$ and $\gamma$. Set $$U=\{ t\in [0,1]: s\mapsto d_{Y}(\sigma(s),h_t(s))\textrm{ is convex} \}.$$ This set is nonempty and closed. Let us show it is also open. If $t_0\in U$ take $\delta$ so small that $$d_\infty(h_t,h_{t_0})<\varepsilon-d_{\widehat Y}(\sigma,h_{t_0})$$ whenever $|t-t_0|<\delta$. Since $d_Y(\sigma,h_{t_0})$ is convex we have the estimate $d_\infty(\sigma,h_{t_0})\le d_{Y^2}(\phi(\sigma),\phi(h_{t_0}))$ and so $$ d_\infty(\sigma,h_t)\le d_\infty(\sigma,h_{t_0})+d_\infty(h_{t_0},h_t)<d_{Y^2}(\phi(\sigma),\phi(h_{t_0}))+\varepsilon- d_{\widehat Y}(\sigma,h_{t_0})<\varepsilon. $$ Consequently $s\mapsto d_Y(\sigma(s),h_t(s))$ is convex. Thus $U$ is open whence $U=[0,1]$, and therefore $d_Y(\sigma,h_1)=d_Y(\sigma,\gamma)$ is convex.

Conversely given any pair $(x,y)\in B_{Y^2}(\phi(\sigma), \varepsilon)$ Lemma \ref{hololemma} yields a unique local geodesic $\gamma\in \widehat Y$ joining $x$ and $y$ such that $t\mapsto d_Y(\sigma(t),\gamma(t))$ is convex. We conclude that $p:B_{\widehat Y}(\sigma,\varepsilon)\to B_{Y^2}(\phi(\sigma),\varepsilon)$ is a bijective $1$-Lipschitz map.

Given $\alpha,\beta\in B_{\diaco Y}(\sigma,\varepsilon)$ take a path $\gamma=(\gamma_0,\gamma_1)$ in $B_{Y^2}(\phi(\sigma),\varepsilon)$ joining $\phi(\alpha)$ and $\phi(\beta)$ and lift it to a path $h$ in $Y\times Y$ in the following way. For each $t$ let $h_t$ be the unique local geodesic joining  $\gamma_0(t)$ and $\gamma_1(t)$ with $s\mapsto d_Y(\sigma(s),h_t(s))$ convex, given by Lemma \ref{hololemma}. Then $h$ is a lift of $\gamma$ joining $\alpha$ and $\beta$ (by the uniqueness) and $d_{\widehat Y}(\alpha,\beta)\le \ell(\phi\circ h)=\ell(\gamma)$. Taking infimum over $\gamma$ we obtain $$d_{\widehat Y}(\alpha,\beta)\le d_{Y^2}(\phi(\alpha),\phi(\beta)).$$ This finishes the proof.
\end{proof}
\end{lemma}

\bigskip\noindent Since $Y^2$ is locally convex (nonpositively curved) it follows that $\diaco Y$ is locally convex (nonpositively curved) and, by \cite[Proposition I.3.28]{bri99} $\phi$ is a covering map.

\bigskip\noindent If $Y$ is locally compact it follows from the Hopf-Rinow theorem that $\diaco Y$ is a complete, proper geodesic space. In the event that $\alpha,\beta\in \tilde Y_q$ (see the discussion after Theorem \ref{universal}) we have $$d_{\widehat Y}(\alpha,\beta)\le d_q(\alpha,\beta).$$ Indeed the identity map $\iota: (\tilde Y_q,d_q)\to (\tilde Y_q, d_{\widehat Y})$ is a local isometry: for every $\alpha\in \tilde Y_q$ the restriction $\iota|_{B_q(\alpha,\varepsilon)}$ is a surjective isometry whenever $\varepsilon >0$ is such that  $B_Y(\alpha(t),\varepsilon)$ is a convex neighbourhood for all $t\in [0,1]$.

\bigskip\noindent A fact we shall use is that, for $\alpha,\beta\in \tilde Y_q$ the distance in the $d_q$ metric is given by $$d_q(\alpha,\beta)=\ell(\langle \alpha\beta^{-1}\rangle),$$ where $\langle \alpha\beta^{-1}\rangle$ denotes the unique local geodesic homotopic to $\alpha\beta^{-1}$. Define $s\mapsto h_s:[0,1]\to \tilde Y_q$ by
\begin{align*}
&h_s (t)=\left\{
\begin{array}{ll}
\beta((1-2s)t) & 0\le s\le 1/2, \\
\alpha((2s-1)t) & 1/2\le s\le 1
\end{array}
\right.
\end{align*}
The path $h$ is the lift of $\alpha\beta^{-1}$ starting at $\beta$ (and ending at $\alpha$). Since $\tilde Y_q$ is simply connected $h$ is homotopic to the unique geodesic $\gamma$ between $\beta$ and $\alpha$. Consequently $\alpha\beta^{-1}$ is homotopic to the local geodesic $p_q\circ\gamma$. By Proposition \ref{unigeo} (uniqueness of local geodesic in the homotopy class of $\alpha\beta^{-1}$) we have $p_q\circ\gamma=\langle \alpha\beta^{-1}\rangle$ and thus $$ d_q(\alpha,\beta)=\ell(\gamma)=\ell(p_q\circ\gamma)=\ell(\langle \alpha\beta^{-1}\rangle).$$

\subsection{Lifts}

\begin{definition}\label{lift}
Let $H:u\simeq v$ be a locally geodesic $p$-quasihomotopy between two maps $u,v\in \Di pXY$. The \emph{lift} $\hat H$ of $H$ is the map $\hat H: X\to \hat Y$ given by mapping $x\in X$ to the local geodesic path $\left(t\mapsto H_t(x)\right)\in \hat Y$.
\end{definition}

\noindent The covering map $p: \hat Y\to Y^2$ also induces a map $p:\Di pX{\hat Y}\to \Di pXY ^2$, \[ p F (x)= (F_0(x),F_1(x)). \] The fact that each component $F_0,F_1\in \Di pXY$ follows from the fact that $p$ is a Lipschitz map. Note that, if $H:u\simeq v$ is a locally geodesic $p$-quasihomotopy and $\hat H$ its lift, the identity $p\circ \hat H=(u,v)$ holds.

\begin{proposition}\label{target}
Let $H:u\simeq v$ be as in Definition \ref{lift}. Then $\hat H\in \Di pX{\hat Y}$ with \[1/2\ (g_u+g_v)\le g_{\hat H}\le g_u+g_v. \]
\begin{proof}
As in the proof of Theorem \ref{di} let us take a sequence $E_m\supset E_{m+1}$ of open sets in $X$ with $\Cp _p(E_m)<2^{-m}$ and $H|_{X\setminus E_m\times [0,1]}$ continuous homotopy between $u|_{X\setminus E_m}$ and $v|_{X\setminus E_m}$, and a path family $\Gamma$ with $\Mod_p(\Gamma)=0$ so that whenever $\gamma\notin \Gamma$,
\begin{enumerate}
\item there exists $m_0$ so that $\gamma^{-1}(E_{m_0})=\varnothing$ (and consequently $\gamma^{-1}(E_m)=\varnothing$ for every $m\ge m_0$ )
\item the inequalities
\begin{align*}
d_Y(u(\gamma(b)),u(\gamma(a)))\le \int_{\gamma|_{[a,b]}}g_u \\
d_Y(v(\gamma(b)),v(\gamma(a)))\le \int_{\gamma|_{[a,b]}}g_v
\end{align*}
hold for $a,b\in [0,1]$.
\end{enumerate}
Let $\gamma \notin \Gamma$ and let $K\subset X\setminus E_{m_0}$ be a compact set containing the image of $\gamma$. Since $H|_{K\times [0,1]}$ is uniformly continuous there is some $\varepsilon >0$ so that $B(z,2\varepsilon)\subset Y$ is a convex ball for all $z\in H(K\times [0,1])$ (the image being a compact set). By the uniform continuity of $H|_{K\times [0,1]}$ there is $\delta>0$ so that whenever $a,b\in [0,1]$ are such that $|a-b|<\delta$ we have $d_Y(H_t(\gamma(b)),H_t(\gamma(a)))<\varepsilon$ for all $t\in [0,1]$. It follows that \[ d_{\hat Y}(\hat H(\gamma(b)),\hat H(\gamma(a)))=d_{Y^2}(p\circ \hat H(\gamma(b)),p\circ \hat H(\gamma(a))) \] for $a,b\in [0,1]$ with $|a-b|<\delta$. On the other hand 
\begin{align*}
&\max\{d_Y(u(\gamma(b)),u(\gamma(a))),d_Y(v(\gamma(b)),v(\gamma(a)))\}\\
&\le d_{Y^2}(p\circ \hat H(\gamma(b)),p\circ \hat H(\gamma(b))) \\
&\le d_Y(u(\gamma(b)),u(\gamma(a)))+d_Y(v(\gamma(b)),v(\gamma(a))).
\end{align*}
From this we see, as in the proof of Theorem \ref{di}, that $g_{\hat H}\le g_u+g_v$. To arrive at the other inequality note that by the leftmost inequality above, any $p$-weak upper gradient for $\hat H$ is also a $p$-weak upper gradient for both $u$ and $v$. Thus $g_u\le g_{\hat H}$ and $g_v\le g_{\hat H}$ almost everywhere, from which we have $$1/2\ (g_u+g_v)\le g_{\hat H}.$$
\end{proof}
\end{proposition}

\noindent Using the map $p$ introduced after Definition \ref{lift} we also have a converse result.
\begin{proposition}\label{converse}
Suppose $F\in \Di pX{\hat Y}$ and $p\circ F=(u,v)$. Then $H(x,t)=F_t(x)$ defines a locally geodesic $p$-quasihomotopy $H:u\simeq v$.
\begin{proof}
By definition for $p$-quasievery $x\in X$ the path $t\mapsto H_t(x)=F_t(x)$ is a local geodesic. Suppose $\varepsilon >0$ is given, and let $E\subset X$ be an open set with $\Cp_p(E)<\varepsilon$ so that $F|_{X\setminus E}$ is continuous. We claim that $H|_{X\setminus E\times [0,1]}$ is a continuous homotopy between $u|_{X\setminus E}$ and $v|_{X\setminus E}$.

From the fact that $p\circ F (x)=(u(x),v(x))$ it is clear that $H|_{X\setminus E\times [0,1]}$ connects $u|_{X\setminus E}$ and $v|_{X\setminus E}$. To see continuity let $(x,t)\in X\setminus E\times [0,1]$ and $\delta>0$ be arbitrary. Choose $\delta_0< \delta$ so that $p:B(F(x),\delta_0)\to B(p\circ F(x),\delta_0)$ is an isometry and, moreover, $B_Y(F_t(x),2\delta_0)$ is a convex ball in $Y$ for every $t\in [0,1]$. By the continuity of $F|_{X\setminus E}$ we find $r>0$ so that $d_{\hat Y}(F(x),F(y))<\delta_0$ whenever $y\in B(x,r)\setminus E$. These choices ensure that the distance function $$ t\mapsto d_Y(F_t(x),F_t(y))$$ is convex (see remark after Lemma \ref{hololemmar}), in particular $$d_\infty (F(x),F(y))\le d_{Y^2}(p\circ F(x),p\circ F(y)). $$ Let us use this to estimate
\begin{align*}
d_Y(H(x,t), H(y,s))& \le d_Y(H(x,t), H(x,s))+d_Y(H(x,s), H(y,s))\\
\le |t-s|\ell(F(x))&+ d_\infty (F(x),F(y))\\
\le |t-s|\ell(F(x))&+ d_{Y^2}(p\circ F(x),p\circ F(y))\\
=|t-s|\ell(F(x))&+ d_{\hat Y}(F(x),F(y))<|t-s|\ell(F(x))+\delta.
\end{align*}
Therefore whenever $(y,s)\in B(x,r)\setminus E\times B(t,\delta/\ell(F(x)))$ we have $$d_Y(H(x,t), H(y,s))<2\delta.$$ Since $\delta>0$ was arbitrary we have the desired continuity.
\end{proof}
\end{proposition}

\begin{remark}\label{repath}
Propositions \ref{target} and \ref{converse} prove Theorem \ref{ppath}; they demonstrate a one-to-one correspondence between locally geodesic $p$-quasihomotopies between maps that are in $\Di pXY$, and elements in $\Di pX{\hat Y}$. Any locally geodesic $p$-quasihomotopy $H$ lifts to a map $\hat H\in \Di pX{\hat Y}$ and, conversely, any map $F\in \Di pX{\hat Y}$ yields a locally geodesic $p$-quasihomotopy.
\end{remark}

\section{The manifold case} In this section we take a look at the situation when the spaces $X,Y$ are compact Riemannian manifolds. We shall adopt the notation $M$ for the domain manifold, and $N$ for the target.

We begin by recalling some definitions relevant to the manifold setting.

\subsection{$([p]-1)$-homotopy}
Let $M$ be an $m$-dimensional Riemannian manifold, $P$ an $a$-dimensional \emph{parameter space} (also a Riemannian manifold) and $D\subset Q$ a domain with compact closure and Lipschitz boundary in a Riemannian manifold $Q$ of dimension $d$. Assume, moreover, that the dimensions satisfy $d+a\ge m$. 

Given a Lipschitz map $H:\overline D\times P\to M$ we denote by $H_\xi: \overline D\to M$ the map $H_\xi(x)=H(x,\xi)$. We further assume that $H$ satisfies

\begin{itemize}
\item[(H1)]\label{H1} $\LIP(H_\xi)\le c_0$ for all $\xi\in P$.
\item[(H2)]\label{H2} There exists a positive number $c_1$ so that the $m$-dimensional Jacobian $J_H$ satisfies $J_H(x,\xi)\ge c_1$ for $\Ha^{d+a}$-almost every $(x,\xi)\in\overline D\times P$.
\item[(H3)]\label{H3} There is a positive number $c_2$ so that $\Ha^{d+a-m}(H^{-1}(y))\le c_2$ for $\Ha^m$-almost every $y\in M$.
\end{itemize}

The following very useful lemma can be found in \cite[Lemma 3.3]{han03}.
\begin{lemma}
Suppose $H:\overline D\times P\to M$ satisfies (H1)-(H3). Then for any non-negative Borel function $g:M\to [0,\infty]$  we have \[ \int_P\int_Dg(H_\xi(x))\ud\Ha^d(x)\ud\Ha^a(\xi)\le c_1^{-1}c_2\int_Mg\ \ud\Ha^m. \]
\end{lemma}

\bigskip\noindent Let us assume that $1\le p\le m$ and consider a map $u\in \Nem pMN$. If $K$ is a rectilinear cell complex and $H:|K|\times P\to M$ is a map such that $H|_{\Delta\times P}$ satisfies (H1)-(H3) for every cell $\Delta \in K$. We have the following \cite[Lemma 4.3]{han03}.

\begin{lemma}
There exists a Borel set $E\subset P$ with $\Ha^a(E)=0$ such that for all $\xi\in P\setminus E$ we have $u\circ H_\xi\in \mathcal{W}^{1,p}(K;N)$. If $k\in \N$, $0\le k<p$ then the map $\chi=\chi_{k,H,u}: P\to [|K^k|;N]$ defined by \[ \chi(\xi)=[u\circ H_\xi|_{|K^k|}] \] is measurable in the sense that $\chi^{-1}\{ \alpha\}$ for any $\alpha \in [|K^k|;N]$.
\end{lemma}

Here $$|K|=\bigcup_{\underset{\dim\Delta =\dim K}{\Delta\in K}}\Delta$$ and $$K^k=\{\Delta\in K:\dim \Delta \le k \},$$ see also \cite[Section 3 and Section 4]{han03}.

\begin{lemma}(\cite[Lemma 4.7]{han03})\label{const}
If $P$ is further connected and $0\le k\le [p]-1$ then $\chi\equiv const.$ $\Ha^a$-almost everywhere on $P$.
\end{lemma}

\bigskip\noindent Let $\varepsilon_0>0$ be small and $V_{\varepsilon_0}(M)=\{ x\in \R^a: \dist(x,M)<\varepsilon_0 \}$ a tubular neighbourhood of $M$. Denote by $\pi:V_{\varepsilon_0}(M)\to M$ the nearest point projection, which is smooth given small enough $\varepsilon_0$. Given a rectilinear cell-decomposition $h:K\to M$ of $M$, we define
\begin{align*}\label{dhomotopy}
H:|K|\times B^a(\varepsilon_0)\to M, \quad H(x,\xi)=\pi(h(x)+\xi).
\end{align*}
Here $B^a(\varepsilon_0)=\{ \xi\in \R^a: |\xi|<\varepsilon_0 \}$. For any $\Delta \in K$ the map $H|_{\Delta\times B_a(\varepsilon_0)}$ satisfies (H1)-(H3) (with $P=B^a(\varepsilon_0)$ and $\overline D=\Delta$).

\bigskip\noindent Given a rectilinear cell decomposition $h:K\to M$ and $u\in \Wem pMN$, we may consider the constant map $\chi_{[p]-1,H,u}$. We denote this constant by $u_{\sharp, p}(h)$
\begin{definition}(\cite[Definition 4.1]{han03})
Two maps $u,v\in \Wem pMN$ are $([p]-1)$-homotopic if $u_{\sharp ,p}(h)=v_{\sharp ,p}(h)$ for any rectilinear cell decomposition $h:|K|\to M$.
\end{definition}

\bigskip\noindent After introducing the setting used in \cite{han03} we proceed with proving Theorem \ref{pquasi}. The following lemma establishes the measurability of a technical tool that will later be used in the proof.
\begin{lemma}
Suppose $H:\overline D \times P \to M$ satisfies (H1) - (H3) and let $E\subset M$ be open. Then the map $\xi\mapsto \Cp_p(H_\xi^{-1}E)$ is lower semicontinuous.
\begin{proof}
Let $\xi_k\to \xi$ as $k\to\infty$ and $x\in H_\xi^{-1}E$ (i.e. $H(x,\xi)\in E$). Since $E$ is open and $H$ continuous there are open neighbourhoods $U\times V\ni (x,\xi)$ so that $H(U\times V)\subset E$. In particular there exists $j$ so that $x\in H_{\xi_k}^{-1}E$ for all $k\ge j$. In other words we have $$ H_\xi^{-1}E\subset \bigcup_{j\ge 1} \bigcap_{k\ge j} H_{\xi_k}^{-1}E.$$ But from this we may estimate, using the properties of the $p$-capacity, \begin{align*}
\Cp_p(H_\xi^{-1}E)&\le \Cp_p\left(\bigcup_{j\ge 1} \bigcap_{k\ge j} H_{\xi_k}^{-1}E\right) = \lim_{j\to \infty} \Cp_p\left( \bigcap_{k\ge j} H_{\xi_k}^{-1}E\right)\\
&\le \liminf_{j\to \infty} \Cp_p(H_{\xi_j}^{-1}E)
\end{align*}
\end{proof}
\end{lemma}

\subsection{$p$-Quasihomotopic maps are path-homotopic but not vice versa}

\bigskip\noindent The following simple counterexample demonstrates that path-homotopy need not imply $p$-quasihomotopy. Take $M=B^2$, the closed unit ball of the plane, and $N=S^1$. Consider the path- and $p-$quasihomotopy classes of the constant map $1$. 

If $H:1\simeq u$ is a $p$-quasihomotopy, $u\in \Nem p{B^2}{S^1}$, we may take the locally geodesic $p$-quasihomotopy and lift it to obtain a map $\tilde H\in \Nem pX{\diaco{S^1}}$ which has the property that $\phi_0\circ\tilde H=1$ quasieverywhere. Thus $\tilde H(x)\in \widetilde N_1$ for $p$-quasievery $x\in X$ and we may view $\tilde H$ as a map $\tilde H:X\to \widetilde N_1\simeq \R$.

It follows that if $u\in \Nem p{B^2}{S^1}$ is $p$-quasihomotopic to the constant map 1, then it admits a lift $\tilde H\in \Nem p{B^2}{\R}$. Conversely any lift $h\in \Nem p{B^2}{\R}$ of yields a $p$-quasihomotopy $H:1\simeq u$ through \[ H(x,t)=\exp{(2\pi i t\cdot h(x)}). \] (Note that $r\mapsto \exp{(2\pi i r)}$) is the covering map $\R\to S^1$.

Consequently the $p$-quasihomotopy class of the constant map consists precisely of those maps $u\in \Nem p{B^2}{S^1}$ which admit a lift $h\in \Nem p{B^2}{\R}$. However we know that not all maps have this property: for example one can consider the map $u(z)=z/|z|$ when $1<p<2$.

In contrast, if $1\le p<2$ then by \cite[Theorem 0.2]{bre01} the space $\Nem p{B^2}{S^1}$ is path connected. Thus the path-homotopy and $p$-homotopy classes do not always agree.


\bigskip\noindent In contrast to the general case, where some curvature assumption on the target space is needed to pass from $p$-quasihomotopy to path-homotopy (cf. Theorem \ref{polku}), the manifold setting does not require such an assumption. This is the content of Theorem \ref{pquasi}

The proof of Theorem \ref{pquasi} will proceed by showing that $u$ and $v$ are $([p]-1)$-homotopic. It is based on the following lemma.

\begin{lemma}
There is a constant $c<\infty$, depending on the data of (H1) - (H3) and on $p$, so that if $E\subset M$ is open then 
\begin{equation}\label{capslice}
\int_P\Cp_p(H_\xi^{-1}E)\ud\Ha^a(\xi) \le c \Cp_p(E)
\end{equation}
\begin{proof}
Suppose $u\in \Ne pM$ is non-negative with $u|_E\ge 1$. Then for $\Ha^a$-almost every $\xi\in P$, $u\circ H_\xi \in \Ne p{\overline D}$ is non-negative and $u\circ H_\xi|_{H_\xi^{-1}E}\ge 1$, whence $$\Cp_p(H_\xi^{-1}E)\le \|u\circ H_\xi\|_{\Ne p{\overline D}}^p.$$ Note that $$ |\nabla(u\circ H_\xi)|\le \|DH_\xi\||\nabla u(H_\xi)|\le c_0|\nabla u(H_\xi)|$$ almost everywhere. Integrating over $P$ and using this estimate we have
\begin{align*}
\int_P \Cp_p(H_\xi^{-1}E)\ud\Ha^a(\xi)&\le c_0^{p}\int_P\int_{\overline D} (|u(H_\xi(x))|^p+|\nabla u(H_\xi(x))|^p\ud\Ha^d(x)\ud\Ha^a(\xi) \\
&\le c_0^{p}c_1^{-1}c_2 \int_M(|u|^p+|\nabla u|^p)\ud\Ha^m.
\end{align*}
Taking infimum over admissible $u$ gives the claim with $c= c_0^{p}c_1^{-1}c_2$.
\end{proof}
\end{lemma}

\begin{proof}[Proof of Theorem \ref{pquasi}]
Let $F:u\simeq v$ be a $p$-quasihomotopy and let $E_j$ be the open sets such that $\Cp_p(E_j)<1/j$ and $H|_{M\setminus E_j\times [0,1]}$ is a classical homotopy $u|_{M\setminus E_j}\simeq v|_{M\setminus E_j}$. Fix a rectilinear cell decomposition $h:K\to M$ and set $H(x,\xi)=\pi(h(x)+\xi)$. Then for any $\Delta \in K^{[p]-1}$ the restriction of $H$ to $\Delta \times B_a(\varepsilon_0)$ satisfies (H1) - (H3) and by Lemma \ref{capslice} \[ \int_{B_a(\varepsilon_0)}\Cp_p((H_\xi|_{\Delta})^{-1}E_j)\ud\Ha^a(\xi)\le c\Cp_p(E_j). \] Denote $$Z=\{ \xi\in B_a(\varepsilon_0): (H_\xi|_{\Delta})^{-1}E_j\neq \varnothing \textrm{ for all } j\}.$$ We have 
\begin{align*}
\int_{B_a(\varepsilon_0)} \liminf_{j\to\infty}\Cp_p((H_\xi|_{\Delta})^{-1}E_j)\ud\Ha^a(\xi) \le &\liminf_{j\to\infty}\int_{B_a(\varepsilon_0)} \Cp_p((H_\xi|_{\Delta})^{-1}E_j)\ud\Ha^a(\xi)\\
\le &c\liminf_{j\to\infty}\Cp_p(E_j)=0.
\end{align*}
But because $\Cp_p(A)\ge c >0$ for any nonempty $A\subset \Delta$ (since $\dim\Delta \le d<p$) it follows that if $\xi \in Z$ then $\liminf_{j\to\infty}\Cp_p((H_\xi|_{\Delta})^{-1}E_j)>0$. This, however, can happen only on a set of $\Ha^a$-measure zero and so $\Ha^a(Z)=0$.

By this and \cite[Lemma 3.5]{han03} we have that for almost every $\xi$
\begin{itemize}
\item[(i)] $u\circ H_\xi\in \mathcal{W}^{1,p}(K;N)$ and
\item[(ii)]$(H_\xi|_{|K^d|})^{-1}E_j= \varnothing$ for some $j$,
\end{itemize} where $d= [p]-1$. For these $\xi$, the restriction $F\circ H_\xi|_{|K^d|\times [0,1]}$ is a homotopy between $u\circ H_\xi|_{|K^d|}$ and $v\circ H_\xi|_{|K^d|}$. Therefore $u_{\sharp, p}(h)=v_{\sharp, p}(h)$ and we are done.
\end{proof}

\bigskip\noindent We already saw in this subsection that path-homotopic maps need not be $p$-quasihomotopic. However if two maps can be connected by a rectifiable curve then they are $p$-quasihomotopic (Theorem \ref{rect}).

Let us close Section 4 with a proof of Proposition \ref{ohtatop}. The proof is essentially contained in \cite[Lemma 1]{chi07}. We sketch it here for completeness.

\begin{proof}[Proof of Proposition \ref{ohta}]
The first part of the claim is standard and can be found in \cite[Introduction]{haj14} and the references therein.

Since $M$ is compact the standard and Ohta topologies are given by a metric. It suffices to prove that a sequence converging in the Ohta metric also converges in the standard metric.

Consider the Nash embedding of $N$ into some $\R^l$ and recall that $$\Wem pMN=\{ u\in \Wem pM{\R^l}: u(x)\in N \textrm{ a.e. }x\in M \}.$$ Take a sequence $(u_j)\subset \Wem pMN$ converging to $u\in \Wem pMN$ in the Ohta metric. Then $u_j\to u$ in $L^p(M;N)$ and the sequence $(\nabla u_j)$ is bounded in $L^p(M;\R^l)$, hence $\nabla u_j \rightharpoonup \nabla u$ in $L^p(M;\R^l)$.

On the other hand convergence in the Ohta metric implies $$\int_M|\nabla u_j|^p\ud vol\rightarrow \int_M|\nabla u|^p\ud vol$$ as $j\to \infty$. The uniform convexity of $L^p(M;\R^l)$ yields that $$\|\nabla u_j-\nabla u\|_{L^p(M;\R^l)}\to 0$$ as $j\to \infty$. Thus $u_j\to u$ in the standard metric.
\end{proof}

\noindent The caveat here is that even though the Ohta metric gives the same topology it is \emph{not} in general a complete metric, see \cite[Lemma 2]{chi07}. The author thanks the anonymous referee for pointing this out.

\section{$p$-quasihomotopy classes of maps}In this section we always assume that $X$ is a complete space with a doubling measure $\mu$ supporting a weak $(1,p)$-Poincar\'e inequality, and that $Y$ is a complete locally convex metric space. Given a map $v\in \Di pXY$ we want to study the $p$-quasihomotopy class of $v$, denoted $[v]_p$. Ultimately, we are interested in its compactness properties since these are the key to proving existence of energy minimizing maps in a given $p$-quasihomotopy class.

A first observation is that \[ [v]_p=\{ F_1: F\in \Di pX{\diaco Y},\ F_0=v \}. \] This is easy to see using the one-to-one correspondence of $p$-quasihomotopies and maps in $\Di pX{\diaco Y}$ presented above. Let us set \[ H^v=\{ F\in \Di pX{\diaco Y}: F_0=v \}. \] Abusing notation slightly we denote by $\phi: H^v\to [v]_p$ the map $$F\mapsto \phi_1\circ F=F_1$$ induced by the covering map $\phi=(\phi_0,\phi_1):\diaco Y\to Y^2$ (since for $F\in H^v$ the first projection  $F_0=phi_0\circ F=v$ always holds we may disregard it).

\bigskip\noindent Let us introduce some notation. Given a $p$-quasihomotopy $H:u\simeq v$ we denote by $\langle H\rangle :u\simeq v$ the locally geodesic $p$-quasihomotopy associated to $H$, given by Theorem \ref{geodhom}. It is evident that, given two $p$-quasihomotopies $H:u\simeq v$ and $H':v\simeq w$ the conjunction $H'H:u\simeq w$ is a $p$-quasihomotopy, and we may consider the locally geodesic representative $\langle H'H\rangle$. We call this the product of $H'$ and $H'$. The \emph{inverse} $H^{-1}$ of a $p$-quasihomotopy $H: u\simeq v$ is simply the $p$-quasihomotopy $H^{-1}: v\simeq u$ given by $$H^{-1}(x,t)=H(x,1-t).$$

\bigskip\noindent Let $G_v$ denote the set of locally geodesic $p$-quasihomotopies $H:v\simeq v$. The product and inverse defined above turn $G_v$ into a group.



Furthermore the group acts on $H^v$ (from the right): given elements $\sigma\in G_v$ and $F\in H^v$ we set $F.\sigma=\langle F\sigma\rangle$. Indeed, the map $$(F,\sigma)\mapsto F.\sigma :H^v\times G_v\to H^v$$ defines a right group action on $H^v$. This is easily seen: $(F.1)^x=F^x$ for all $F\in H^v$ and $(F.(\sigma_2\sigma_1))^x=\langle F(\sigma_2\sigma_1)\rangle^x= \langle F^x\sigma_2^x\sigma^x\rangle=\langle (F^x\sigma_2^x)\sigma_1^x\rangle= \langle (F\sigma_2)\sigma_1\rangle^x=((F.\sigma_2).\sigma_1)^x$ for $p$-quasievery $x\in X$.

Pointwise, this is the action of $\pi_1(Y,v(x))$ on the universal covering space $\tilde Y_{v(x)}$ (for $p$-quasievery $x\in X$).
\begin{remark}\label{deck}
The group $G_v$ acts on $H^v$ by ``deck transformations'', i.e. \[ \phi\circ (F.\sigma)= \phi\circ F \] for $F\in H^v,\sigma\in G_v$. This is directly seen from the definitions.
\end{remark}
Next we demonstrate that the action of $G_v$ on $H^v$ is in fact both free and proper (in the sense of \cite[Chapter I.8, Definition 8.2]{bri99}). The following definition and lemma will prove useful.

\begin{definition}\label{quasi}
We say that a set $U\subset X$ is $p$-quasiopen ($p$-quasiclosed), or quasiopen (quasiclosed) for short, if, for every $\varepsilon >0$ there exists an open set $E\subset X$ with $\Cp _p(E)<\varepsilon$ so that $U\setminus E$ is open (closed) in $X\setminus E$.
\end{definition}

\begin{lemma}\label{quasidiscrete}
Suppose $X$ is compact, $f\in \Ne pX$ and the set $\{f=0 \}$ both quasiclosed and quasiopen. Then either $$ \Cp_p(\{f=0\})=0$$ or $$\Cp_p(X\setminus \{ f=0 \})=0.$$
\end{lemma}
\begin{proof}
Set $A=\{ f=0 \}$. Let $F_n\supset F_{n+1}$ be a decreasing sequence of open sets in $X$ such that $\Cp_p(F_n)<2^{-n}$, $f|_{X\setminus F_n}$ is continuous and $A\setminus F_n$ is both closed and open in $X\setminus F_n$. We further denote by $F$ the intersection of all $F_n$'s.

Suppose $\Cp_p(A)>0$. Then also $\Cp_p(A\setminus F)>0$. First we will show that $\mu(X\setminus A)=0$. If $\mu(X\setminus A)>0$ then also $\mu(X\setminus (F\setminus A))>0$. Since, for given $n\in\N$ the set $A\setminus F_n$ is both closed and open in $X\setminus F_n$ the same is true of $X\setminus (A\cup F_n)=(X\setminus A)\setminus F_n$. Therefore the sets $A\setminus F_n$ and $X\setminus (A\cup F_n)$ form a separation of $X\setminus F_n$, for all $n$.

Take $x\in A\setminus F$ and $y\in X\setminus (A\cup F)$ with $\Cp_p(B(x,r)\setminus F)>0$ and $\mu(B(y,r)\setminus (A\cup F))>0$ for all $r>0$. The condition is automatic for $x$ since by \cite[Theorem 6.7 (xii)]{bjo11} $\Cp_p(B(x,r))=\Cp_p(B(x,r)\setminus F)$ and it is true for $y$ provided we choose $y$ to be a density point of $X\setminus (A\cup F)$. 

Take $0<r<d(x,y)/2$ whence $\overline B(x,r)\cap \overline B(y,r)=\varnothing$. Furthermore the sets $B_n=(A\cap \overline B(x,r))\setminus F_n$, $B_n'=\overline B(y,r)\setminus (A\cup F_n)$ are disjoint and compact. Now \cite[Theorem 7.33]{hei01} implies (for compact $X$!) \[ \Mod_p(\Gamma_n)=\Cp_p(B_n,B_n')=\inf\left\{ \int_Xg_u^p\ud\mu; u\in \Ne pX,\ u|_{B_n}\equiv 0, u|_{B_n'}\ge 1 \right\} \] where $\Gamma_n=\Gamma_{B_n,B_n'}$ \footnote{Here $\Gamma_{U,V}$ denotes the path family connecting the sets $U$ and $V$}. Choose a ball $B_0\subset X$ so that $B_n\cup B_n'\subset \overline B(x,r)\cup\overline B(y,r)\subset B_0$ and estimate, for any $u$ as in the above infimum by \cite[Theorem 5.53]{bjo11}
\begin{align*}
&\frac{\mu(B_n')}{\mu(2B_0)}\le \dashint_{2B_0}|u|^p\ud\mu\le \frac{C}{\Cp_p(B_0\cap \{ u=0\})}\dashint_{2\sigma B_0}g_u^p\ud\mu \le \frac{C/\mu(2B_0)}{\Cp_p(B_n)}\int_Xg_u^p\ud\mu
\end{align*}
from which we get, taking infimum over $u$, \[ \Mod_p(\Gamma_n)\ge 1/C\mu(B_n')\Cp_p(B_n). \] Since $\displaystyle (A\cap \overline B(x,r))\setminus F=\bigcup_n B_n,\ \overline B(y,r)\setminus (A\cup F)=\bigcup_nB_n'$ we have $$\lim_{n\to \infty} \mu(B_n')\Cp_p(B_n)= \mu(\overline B(y,r)\setminus (A\cup F))\Cp_p((A\cap \overline B(x,r))\setminus F) $$ and thus \[ \Mod_p(\Gamma_0)\ge \limsup_{n\to \infty}\Mod_p(\Gamma_n)\ge \alpha>0, \] where $$\Gamma_0=\Gamma_{\overline B(y,r)\setminus (A\cup F),(A\cap \overline B(x,r))\setminus F }.$$ Let $\Gamma_\infty=\{ \gamma: \gamma^{-1}(F_n)\neq\varnothing \forall n \}$ whence by Lemma \ref{refinement} $\Mod_p(\Gamma_\infty)=0$. From the fact that $$\Mod_p(\Gamma_0\setminus \Gamma_\infty)\ge \Mod_p(\Gamma_0)-\Mod_p(\Gamma_\infty)\ge \alpha>0$$ we conclude that there exists a curve $\gamma \in \Gamma_0\setminus \Gamma_\infty$. In other words there exists an index $n_0$ and a curve $\gamma\in \Gamma_0$ with $|\gamma|\subset X\setminus F_{n_0}$. Such a curve joins the sets $A\setminus F_{n_0}$ and $X\setminus (A\cup F_{n_0})$ in $X\setminus F_{n_0}$. This, however should be impossible since these two sets separate $X\setminus F_{n_0}$. 

\bigskip\noindent We conclude that $\mu(X\setminus A)=0$, that is, $f= 0$ almost everywhere. Since $f$ is $p$-quasicontinuous it follows \cite[Proposition 1.59]{bjo11} that $f=0$ $p$-quasieverywhere, i.e $\Cp_p(X\setminus A)=0$. The proof is now complete.
\end{proof}

\bigskip\noindent This lemma will be used to prove that the projection $\phi:H^v\to [v]_p$ is a discrete map. Namely we have 

\begin{proposition}\label{discrete}
Suppose $X$ and $Y$ are compact and let $u\in [v]_p$. Then the set \[ \phi^{-1}(u)=\{ F\in H^v: F_1=u \} \] is discrete with respect to the metric \[ \hat d(F,H):=\left(\int_Xd_{\hat Y}^p(F,H)\ud\mu\right)^{1/p}. \]
\end{proposition}
\begin{proof}
Suppose $H, F\in \phi^{-1}(u)$ are distinct and let $\sigma = \langle HF^{-1}\rangle$ be the locally geodesic $p$-quasihomotopy $u\simeq u$ in the q.e pointwise homotopy class of $HF^{-1}:u\simeq u$.

Consider the map $l_\sigma\in \Ne pX$, given by $$l_\sigma(x)=\ell(\sigma(x)).$$ Let $\varepsilon_Y$ be half the injectivity radius of $Y$, i.e. the largest number $r$ with the property that every ball $\overline B(y,2r)$, $y\in Y$, is a Busemann space. This is positive since $Y$ is compact. It follows that if $l_\sigma(x)<\varepsilon_Y$ then the loop $\sigma(x)$ is contractible, by Busemann convexity. Therefore we have \[ \{x\in X: l_\sigma(x)<\varepsilon_Y \} = \{ x\in X: l_\sigma(x)=0 \}=:U. \] Since $l_\sigma$ is $p$-quasicontinuous  it follows that $U$ is both $p$-quasiclosed and $p$-quasiopen, whence by Lemma \ref{quasidiscrete} either $\Cp_p(U)=0$ or $\Cp_p(X\setminus U)=0$. 

Note further that $$ U=\{ x\in X: d_{\hat Y}(F(x),H(x))<\varepsilon_Y \}.$$ This is because for any $q\in Y$ the inclusion map $\iota_q: (\tilde Y_q, d_q)\to (\hat Y, d_{\hat Y})$ is a local isometry with every restriction $\iota|_{B(\alpha, \varepsilon_Y)}$, $\alpha\in \tilde Y_q$, an isometry (see the discussion after the construction of $\widehat Y$, Section 3). This in turn implies $$ d_{\hat Y}(F(x),H(x))=d_{v(x)}(F(x),H(x))=l_\sigma(x)$$ whenever $l_\sigma(x)<\varepsilon_Y$ (or equivalently $d_{\hat Y}(F(x),H(x))<\varepsilon_Y$), yielding the desired identity.

\bigskip\noindent Now suppose that $\hat d(F,H):=\varepsilon <\varepsilon_Y\mu(X)^{1/p}$. Then we have \[ \mu(\{ x\in X: d_{\tilde Y}(F(x),H(x))\ge \varepsilon_Y \})\le \left(\frac{\varepsilon}{\varepsilon_Y}\right)^p <\mu(X),\] implying \[ \mu(U)= \mu(X)-\mu(\{ x\in X: d_{\tilde Y}(F(x),H(x))\ge \varepsilon_Y \})>0. \] By Lemma \ref{quasidiscrete} we therefore have $\Cp_p(X\setminus U)=0$, in other words $l_\sigma=0$ $p$-quasieverywhere which implies $\hat d(F,H)=0$. 

This, however is not possible since $F$ and $H$ are distinct and therefore we conclude that any two distinct $F,H\in \phi^{-1}(u)$ must satisfy $$\hat d(F,H)\ge \varepsilon_Y\mu(X)^{1/p}.$$

\end{proof}

\bigskip\noindent We now introduce two minor alterations to the discussion above. The first one is a change of metric; for us it is convenient to use the metric \[ \tilde d(F,H)^p=\int_Xd_{v(x)}^p(F(x),H(x))\ud\mu(x) \] on $H^v$ instead of $\hat d$. This way we ensure that $G_v$ acts on $H^v$ by isometries. Indeed for $p$-quasievery $x\in X$ we have $$d_{v(x)}(\langle F(x)\sigma(x)\rangle, \langle H(x)\sigma(x)\rangle)=d_{v(x)}(H(x),F(x)),$$ $F,H\in H^v$, $\sigma\in G_v$, since pointwise this is simply the action of $\sigma(x)\in \pi(Y,v(x))$ on $\widetilde Y_{v(x)}$ by isometry. It follows that \[ \tilde d(F\sigma,H\sigma)=\tilde d(F,H). \] From the elementary inequality \[ d_{\hat Y}(\alpha,\beta)\le d_q(\alpha,\beta)\quad \alpha,\beta\in \tilde Y_q, \] it follows that \[ \hat d\le\tilde d, \] in particular the claim of Proposition \ref{discrete} remains true if the metric $\hat d$ is replaced by $\tilde d$.

\medskip\noindent The second alteration is on the space $H^v$. We introduce a parameter $M\in (0,\infty]$ and denote by $H^v_M$ the set \[ H^v_M= \{ F\in H^v: \|g_{\phi\circ F}\|_{L^p}\le M \}. \] 

In other words we restrict our attention to maps $F\in H^v$ for which the endpoint $u=\phi\circ F$ satisfies a gradient $L^p$-norm upper bound. By Theorem \ref{di}, if $u\in \Nem pXY$, $M\ge \|g_v\|_{L^p}, \|g_u\|_{L^p}$, and $H:v\simeq u$ is a locally geodesic $p$-quasihomotopy then $\|g_{H_t}\|_{L^p}\le M$ for every $t$. Clearly the claim of Proposition \ref{discrete} remains true if, in addition to the change of metric, the space $H^v$ is replaced by $H^v_M$. We use the notation $[v]_{p,M}$ for the image set $\phi(H^v_M)\subset [v]_p$.

A little care is needed when considering $H^v_M$ as a metric space with either of the metrics $\hat d$ or $\tilde d$, since these only measure differences of maps up to sets of measure zero. A crucial observation is that if $u,v:X\to Y$ admit $p$-integrable upper gradients and $u=v$ almost everywhere, then in fact $u=v$ $p$-quasieverywhere and they may regarded as the same element in $\Di pXY$. This may be seen by applying \cite[Proposition 1.59]{bjo11} to $d(u,v)$ and $0$.

\begin{lemma}\label{proper}
The set $H^v_M$ equipped with the metric $\tilde d$ is a proper metric space.
\end{lemma}
\begin{proof}
Take a sequence $$F_n\in \tilde B(H, L):= \{ F\in H^v_M: \tilde d(H,F)\le L \} .$$ Each $F_n$ is the lift of the locally geodesic $p$-quasihomotopy $t\mapsto (F_n)_t: v\simeq \phi\circ F_n$, so using Proposition \ref{target} we may estimate
\begin{align*}
&\hat d(H,F_n)+\|g_{F_n}\|_{L^p}\le L+ \|g_v\|_{L^p}+\|g_{\phi\circ F_n}\|_{L^p}\le \|g_v\|_{L^p}+L+M
\end{align*}
for all $n$ and therefore the Rellich Kondrakov theorem \ref{rellich} implies that a subsequence denoted $F_n$ converges to some $F\in \Nem pX{\diaco Y}$ in the metric $\hat d$. By passing to a further subsequence we may assume that $F_n\to F$ pointwise almost everywhere. (In particular $d_{v(x)}(F_n(x),F(x))\to 0$ as $n\to\infty$ for almost every $x\in X$.)

From the fact that $d_q(\alpha,\beta)= \ell(\langle\beta\alpha^{-1}\rangle)$ for paths $\alpha,\beta \in \widetilde Y_q$ (see discussion before Definition \ref{lift}) we observe that $l_{F_n}(x)=d_{v(x)}(F_n(x),\hat v)$, where $\hat v$ denotes the lift of the trivial $p$-quasihomotopy $v\simeq v$. Using this and Lemma \ref{locint} we may estimate
\begin{align*}
\|l_{F_n}\|_{L^p}+\|g_{l_{F_n}}\|_{L^p}&= \tilde d(F_n,\hat v)+\|g_{l_{F_n}}\|_{L^p}\\
&\le \tilde d(\hat v, H)+\tilde d(H, F_n)+\|g_{p\circ F_n}\|_{L^p}+\|g_v\|_{L^p}\\
&\le \tilde d(\hat v, H)+L+M+\|g_v\|_{L^p}
\end{align*}
for all $n$, so for a still further subsequence the function $l_{F_n}$ converges to some $f\in \Ne pX$ in $L^p$-norm and pointwise almost everywhere. We shall use the following General Lebesgue Dominated Convergence Theorem.
\begin{lemma}
Let $f_n$ be a sequence of measurable functions on a measure space $(\Omega, \nu)$ that converges $\nu$-almost everywhere to $f$. Suppose there is a sequence $g_n$ of $\nu$-integrable functions that converge pointwise $\nu$-almost everywhere to a $\nu$-integrable function $g$, such that $|f_n|\le g_n$ for each $n$, and $$\lim_{n\to \infty}\int_\Omega g_n\ud\nu=\int_\Omega g\ud\nu.$$ Then $$\lim_{n\to\infty}\int_\Omega f_n\ud\nu = \int_\Omega f\ud\nu.$$
\end{lemma} 

\noindent By the inequality \[ d_v^p(F_n,F)\le 2^{p-1}d_v^p(F_n,\hat v)+2^{p-1}d_v^p(\hat v,F)=2^{p-1}l_{F_n}^p+2^{p-1}l_F^p \] we may take $g_n=2^{p-1}l_{F_n}^p+2^{p-1}l_F^p$ and $g=2^{p}l_F^p$ and use the above theorem to conclude
\[ \lim_{n\to\infty}\int_Xd_v^p(F_n,F)\ud\mu=\int_X\lim_{n\to\infty} d_v^p(F_n,F)\ud\mu=0. \] Having established $\tilde d(F_n,F)\to 0$ as $n\to\infty$ it is evident that $F\in \tilde B(H,L)$ and therefore we have shown the compactness of $\tilde B(H,L)$.
\end{proof}

\noindent The next lemma expresses some nice properties of the action of the group $G_v$ on $H^v_M$ ($M\ge \|g_v\|_p$).

\begin{proposition}\label{proper}
The action of $G_v$ on $H^v_M$ is proper and free. Moreover, if $F\in H^v_M$ and $u=\phi\circ F\in [v]_{p,M}$ then $F.G_v=\phi_M^{-1}(u)$. Here $\phi_M=\phi|_{H^v_M}$.
\end{proposition}
\begin{proof}
Let us first show that the action is free. If $F\sigma=H\sigma$ then for $p$-quasievery $x\in X$ one has $\langle F(x)\sigma(x)\rangle = \langle H(x)\sigma(x)\rangle$. Since the action of $\pi_1(Y,v(x))$ on $\tilde Y_{v(x)}$ is free this implies that $\sigma(x)$ is the neutral element of $\pi_1(Y,v(x))$, i.e the constant path $v(x)$. Since $\sigma(x)$  is the path $t\mapsto v(x)$ for $p$-quasievery $x\in X$ we have that $\sigma$ is the trivial $p$-quasihomotopy $v\simeq v$, i.e. $\sigma_t=v$ for all $t\in [0,1]$.

Now suppose $H\in B(F,\varepsilon)\cap B(F\sigma,\varepsilon)$. Then $\tilde d(F,F\sigma)\le 2\varepsilon$. By Remark \ref{deck} $\phi\circ F=\phi\circ (F\sigma)$, and thus Proposition \ref{discrete} and its proof implies that if $2\varepsilon <\varepsilon_Y\mu(X)^{1/p}=:\varepsilon_0$ then $F=F\sigma$, i.e $\sigma$ is the trivial $p$-quasihomotopy $v\simeq v$, the neutral element of the group $G_v$. This shows that for $\varepsilon<\varepsilon_0/2$ the collection of $\sigma\in G_v$ for which $B(F,\varepsilon)\cap B(F\sigma,\varepsilon)\ne \varnothing$ consists only of the neutral element.

Finally let $F\in H^v_M$ and $u=\phi\circ F\in [v]_{p,M}$. Obviously $F.G_v\subset \phi_M^{-1}(u)$ since for all $\sigma\in G_v$ it holds that $\phi\circ F\sigma=p\circ F$. But if $H\in \phi_M^{-1}(u)$, let $\sigma =\langle F^{-1}H\rangle\in G_v$ and calculate $F\sigma=\langle FF^{-1}H\rangle =H$ so that $H\in F.G_v$.
\end{proof}



\section{A weak compactness result and further discussion} Unfortunately I have been unable to prove that the (restricted) $p$-quasihomotopy class $[v]_{p,M}$ is compact with respect to the $L^p$-metric $\tilde d$.

To look for weaker results we shall utilize the metric properties of the spaces $(H^v_M, \tilde d)$ and $(H^v_M,\hat d)$. \emph{In this section we assume that $X$ and $Y$ are both compact.} In particular $\diaco Y$ is then proper.

An immediate corollary of Lemma \ref{proper} is the following weak compactness result.

\begin{corollary}\label{weak}
Suppose $v\in \Nem pXY$ and $u_n$ is a sequence in $[v]_p$ with \[ \sup_n\|g_{u_n}\|_{L^p}<\infty \] converging to $u$ in $L^p(X;Y)$. If the maps $u_n$ can be connected to $v$ by $p$-quasihomotopies $H_n:v\simeq u_n$ satisfying 
\begin{equation}\label{extra}
\sup_n \int_X l_{H_n}\ud\mu <\infty,
\end{equation}
then $u\in [v]_p$.
\end{corollary}
\begin{proof}
Let $\displaystyle M_0= \sup_n\|g_{u_n}\|_{L^p}$ and $\displaystyle M_1= \sup_n \int_X l_{H_n}\ud\mu$. Using the Poincar\'e inequality we estimate 
\begin{align*}
\tilde d (\hat v, \hat H_n)&= \left(\int_Xl_{H_n}^p\ud\mu\right)^{1/p}\le \dashint_Xl_{H_n}\ud\mu+ C\diam(X)\left(\int_X g_{l_{H_n}}^p\ud\mu\right)^{1/p} \\
&\le \mu(X)^{-1}M_1+C\diam(X)(M_0+\|g_v\|_{L^p})
\end{align*}
(We use the notation $\hat v$ for the lift of the trivial $p$-quasihomotopy $v\simeq v$ again.) Therefore $H_n\in H^v_{M_0}\cap \tilde B(\hat v, L)$, where $L=\mu(X)^{1/p-1}M_1+C\diam(X)M_0$. By the previous lemma a subsequence $H_n$ converges to some $H\in H^v_M$ in the metric $\tilde d$.

Furthermore we have \[ \int_Xd^p_Y(\phi\circ H,u)\ud\mu=\lim_{n\to\infty}\int_Xd^p_Y(\phi\circ H,p\circ  H_n)\ud\mu\le \lim_{n\to\infty}\tilde d (H, H_n)=0\] so that $u=\phi\circ H$. Therefore $u\in [v]_p$.
\end{proof}

\noindent This is an unsatisfactory result because of the extra assumption (\ref{extra}) of having to control the lengths of the homotopies $H_n$. The result is basically a restatement of the fact that, given $M>0$ the space $H^v_M$ equipped with metric $\tilde d$ is a proper (which in turn followed easily from the Rellich Kondrakov compactness theorem \ref{rellich}).

Removing the extra assumption (\ref{extra}) on the homotopies $H_n$ in Corollary \ref{weak} amounts to ensuring that the space $H^v_M/G_v$, arising from the action of $G_v$ on $H^v_M$ in the previous subsection, equipped with the metric $$\overline d(F.G_v,H.G_v):=\dist_{\tilde d}(F.G_v,H.G_v)$$ \emph{has finite diameter}; notice that the action of $G_v$ on $H^v_M$ gives rise to a covering map $$\pi:H^v_M \to H^v_M/G_v$$ and the metric $\overline d$ makes $\pi$ into a local isometry, see \cite[Chapter I.8, Proposition 8.5(3)]{bri99}. 

To see the claim about the finite diameter take two elements $F.G_v, H.G_v\in H^v_M/G_v$ and let $u=\phi\circ F, w=\phi\circ H$. Note that \[ \overline d(F.G_v, H.G_v) = \inf_{\sigma\in G_v} \tilde d(F, \langle H\sigma\rangle)=\inf_{\sigma\in G_v}\left(\int_X l_{\langle H\sigma F^{-1}\rangle}^p\ud\mu\right)^{1/p}. \] The rightmost infimum is equal to the infimum over all locally geodesic $p$-quasihomotopies $H:u\simeq w$ of the quantity $$\left(\int_X l_{H}^p\ud\mu\right)^{1/p},$$ since for each $\sigma \in G_v$, $\langle H\sigma F^{-1}\rangle: u\simeq w$ is a locally geodesic $p$-quasihomotopy. Conversely, given any locally geodesic $p$-quasihomotopy $H': u\simeq w$ we may write it as $H'=\langle H \langle H^{-1}H'F\rangle F^{-1}\rangle$, where $\langle H^{-1}H'F\rangle\in G_v$.

We obtain 
\begin{equation}\label{innermetric}
\overline d(\phi^{-1}(u),\phi^{-1}(w))=\inf_{H:u\simeq w}\left(\int_Xl_H^p\ud\mu\right)^{1/p}.
\end{equation}

With this in hand it is easy to see that if $H^v_M/G_v$ has finite diameter then (\ref{extra}) is automatically satisfied.

On the other hand, requiring that for every sequence $u_n\in [v]_M$ condition (\ref{extra}), rewritten \[ \sup_n\inf_{H:v\simeq u_n}\int_X l_H^p\ud\mu <\infty, \] is satisfied, is equivalent to requiring that there is some $C<\infty$ so that \[\sup_{u\in [v]_M}\inf_{H:v\simeq u}\int_X l_H^p\ud\mu\le C\] (if such a constant did not exist we would have a sequence $u_n$ contradicting the condition). Thus we see that (\ref{extra}) is automatically satisfied if and only if $H^v_M/G_v$ has finite diameter.

Observe that since $(H^v_M,\tilde d)$ is proper the same is true of $H^v_M/G_v$ and therefore it has finite diameter if and only if it is compact.

\vspace{1cm}\noindent What, then, can we say about the quotient space $H^v_M/G_v$? 

We may define a map $\overline \phi_M: H^v_M/G_v\to [v]_{p,M}$ by \[\overline \phi_M(F.G_v)=\phi\circ F.\] This is well-defined by Remark \ref{deck}. By the last assertion in Proposition \ref{proper} we see that $\overline \phi_M$ is bijective.

With the continuous bijection $\overline \phi_M: H^v_M/G_v\to [v]_{p,M}$ at hand it is immediate that compactness of $[v]_{p,M}$ is implied by the compactness of $H^v_M/G_v$; indeed assuming this, the map $\overline \phi_M$ is a homeomorphism (by elementary topological considerations). In this event, furthermore, we see that $\phi_M=\overline \phi_M\circ \pi: H^v_M\to [v]_{p,M}$ is a covering map. (Conversely, assuming that $\phi_M$ is a covering map we have that $\overline \phi_M$ is also a covering map, and therefore a homeomorphism.)

\bigskip\noindent Another way of interpreting the identity (\ref{innermetric}) is to identify $H^v_M/G_v$ with $[v]_{p,M}$ (through the map $\overline \phi_M$) and $\overline d$ with a metric induced by a certain \emph{length structure} on $[v]_{p,M}$ (see \cite[Definition 1.3, p.3]{gro07}). Indeed, the length structure is given by the family of paths $H$ that are (locally geodesic) $p$-quasihomotopies between the endpoint maps, and the length functional is simply $$ \ell(H)=\left(\int_X l_H^p\ud\mu\right)^{1/p}.$$

This point of view emphasizes the (geo)metric structure of $[v]_{p,M}$, or $\Nem pXY$ and in particular the question of compactness of $[v]_{p,M}$ is reduced to asking does the length structure $\overline d$ give rise to the same topology on $[v]_{p,M}$ as does the original $L^p$-metric $$d(u,w)=\left(\int_Xd_Y^p(u,w)\ud\mu\right)^{1/p}.$$ This question remains open, along with the question of existence of energy minimizing maps in $p$-quasihomotopy classes, encouraging the study of geometry of the Newtonian spaces $\Nem pXY$.

\bigskip\subsubsection*{Acknowledgements} I would like to thank Nageswari Shanmugalingam and Tomasz Adamowicz for many helpful comments and suggestions for a clearer exposition. I would also like to thank my advisor Ilkka Holopainen for financial support.

\bibliographystyle{plain}
\bibliography{abib}

\end{document}